\documentclass[12pt]{article}
\usepackage{latexsym,amsmath,amssymb}
\usepackage[dvipsnames,usenames]{color}

\newif\ifdviwin

\usepackage[latin1]{inputenc}
\usepackage[english]{babel}
\usepackage{indentfirst}
\usepackage[mathscr]{eucal}
\usepackage{amssymb,amsmath,amsfonts}
\usepackage{fancybox,fancyhdr}
\usepackage{graphicx}

\newif\ifdviwin

\dviwintrue

\def\cS{\mathcal{S}}

\let\8=\infty \let\0=\emptyset

\let\tilde=\widetilde

\let\landa=\lambda
\let\alfa=\alpha

\let\parc=\partial

\def\landa{\lambda}

\def\flecha{\rightarrow}
\def\esiz{\langle}
\def\esde{\rangle}

\def\cte.{\mathop{\rm cte.}\nolimits}

\def\det{\mathop{\rm det}\nolimits}

\def\cosh{\mathop{\rm cosh }\nolimits}

\def\LO{\mathbb{L}}

\def\R{\mathbb{R}}
\def\Z{\mathbb{Z}}

\def\H{\mathbb{H}}
\def\S{\mathbb{S}}

 \newtheorem{defi}{Definition}
 \newtheorem{teo}[defi]{Theorem}
 \newtheorem{pro}[defi]{Proposition}
 
 \newtheorem{lem}[defi]{Lemma}
 
 \newtheorem{remark}[defi]{Remark}

 \newenvironment{proof}{\rm \trivlist \item[\hskip \labelsep{\it
      Proof}:]}{\par\nopagebreak \hfill $\Box$ \endtrivlist}

\numberwithin{equation}{section}

\begin{document}
\mbox{}\vspace{0.4cm}\mbox{}

\begin{center}
\rule{15.2cm}{1.5pt}\vspace{0.5cm}

{\Large \bf The space of Lorentzian flat tori
\\[0.3cm] in anti-de Sitter $3$-space}\\ \vspace{0.5cm} {\large María A. León-Guzmán
$\mbox{}^a$, Pablo Mira$\mbox{}^b$ and José A. Pastor$\mbox{}^a$}\\
\vspace{0.3cm} \rule{15.2cm}{1.5pt}
\end{center}
  \vspace{1cm}

\noindent $\mbox{}^a$ Departamento de Matem\'{a}ticas, Universidad
de Murcia, Spain \\ e-mails: maleong@um.es, josepastor@um.es
\vspace{.1cm}

\noindent $\mbox{}^b$ Departamento de Matem\'{a}tica Aplicada y
Estadística,
Universidad Polit\'{e}cnica de Cartagena, E-30203 Cartagena, Murcia, Spain. \\
e-mail: pablo.mira@upct.es \vspace{0.2cm}


\vspace{0.3cm}

 \begin{abstract}
We describe the space of isometric immersions from the Lorentz
plane $\LO^2$ into the anti-de Sitter $3$-space $\H_1^3$, and solve
several open problems of this context raised by M. Dajczer and K.
Nomizu in 1981. We also obtain from the above result a description
of the space of Lorentzian flat tori isometrically immersed in
$\H_1^3$ in terms of pairs of closed curves with wave front
singularities in the hyperbolic plane $\H^2$ satisfying some
compatibility conditions.
 \end{abstract}

\section{Introduction}

A classical problem in Lorentzian geometry is the description of the
isometric immersions between Lorentzian spaces of constant
curvature. In this paper we investigate the specific problem of
classifying the isometric immersion from the Lorentz plane $\LO^2$
into the $3$-dimensional anti-de Sitter space $\H_1^3$.

The study of isometric immersions from $\LO^2$ into $\H_1^3$ starts
from a pioneer work by M. Dajczer and K. Nomizu \cite{Dajczer} in 1981.
There, these authors gave a local description of such surfaces in
terms of the Lie group structure of $\H_1^3$, using a classical idea
by L. Bianchi \cite{Bia} to describe the flat surfaces of the Riemannian unit
sphere $\S^3$. Nevertheless, the global problem of finding all
isometric immersions of $\LO^2$ into $\H_1^3$ turned out to be more
subtle than its Euclidean counterpart, and remained open in that
paper. Moreover, Dajczer and Nomizu proposed in \cite{Dajczer} several
specific open problems on the structure of the space of such
isometric immersions from $\LO^2$ into $\H_1^3$, that still remain
unanswered.

In this paper we provide a general description of all isometric
immersions of $\LO^2$ into $\H_1^3$ in terms of pairs of curves with
singularities (\emph{wave fronts}) in the hyperbolic plane $\H^2$.
In particular, we give an answer to the open problems proposed in
\cite{Dajczer}. In order to do so, we adapt to the Lorentzian setting
an important idea by Y. Kitagawa \cite{Kit1} used to describe
complete flat surfaces in $\S^3$ via the Hopf fibration. The main
difficulty in such an adaptation is that, in the Lorentzian case,
the asymptotic curves of a timelike flat surface have varying causal
character. This is a substantial complication in proving that an
isometric immersion of $\LO^2$ into $\H_1^3$ can be globally
parametrized by asymptotic curves, which is the key idea of the
Riemannian case.

An important fact in the context we are working is that, among all
isometric immersions of $\LO^2$ into $\H_1^3$, some of them are
actually universal coverings of immersed (and sometimes embedded)
Lorentzian tori in $\H_1^3$. The basic examples in this sense are
the \emph{Hopf tori} constructed in \cite{amarillo2,amarillo} by means of the
Hopf fibration of $\H_1^3$ over $\H^2$.

The existence of Lorentzian flat tori in $\H_1^3$ is a very
remarkable fact since, in the Lorentzian context, there are very
severe restrictions for the existence of compact immersed Lorentzian
surfaces in an ambient Lorentzian $3$-manifold. Indeed:

\begin{enumerate}
\item
Even intrinsically, any compact surface that admits a Lorentzian metric must be
homeomorphic to a torus (by the Poincaré-Hopf index theorem, see for
instance \cite{ONe}).
 \item
If in a Lorentzian $3$-manifold there exists an immersed compact
Lorentzian surface, then such a $3$-manifold cannot be
\emph{chronological} (i.e. it has to admit closed timelike
curves). The reason is that any compact Lorentzian surface must have closed timelike curves, see \cite[Theorem 3.6]{MiSa}. In particular, there are no compact Lorentzian surfaces
in the universal covering of $\H_1^3$ (which is the unique complete simply-connected
Lorentzian $3$-manifold of constant curvature $-1$).
\end{enumerate}

These results show that, in fact, the case of Lorentzian flat tori
in $\H_1^3$ can be seen as one of the most geometrically simple situations
in which compact Lorentzian surfaces exist inside a Lorentzian
$3$-manifold.

Our second main objective here is to describe the space of
Lorentzian flat tori in $\H_1^3$, as an application of our previous
description of all isometric immersions of $\LO^2$ into $\H_1^3$. In
particular, we prove that any such torus can be recovered in terms
of two closed curves in $\H^2$, one of them regular and the other
one possibly having wave-front singularities. This result can be
seen as an extension to the Lorentzian setting of Kitagawa's
classification of (Riemannian) flat tori of the unit sphere $\S^3$
\cite{Kit1}, although there are several technical differences in the
proof and the final classification theorem. For results about complete flat surfaces in $\S^3$, we may refer the reader to \cite{Kit1,Kit2,Wei,GaMi,AGM2,DaSh} and references therein.

We have organized this paper as follows. In Section 2 we give some
preliminaries on the geometry of $\H_1^3$ as a Lie group by means of a pseudo-quaternionic structure, and we
introduce the different Hopf fibrations existing on $\H_1^3$. In
Section 3 we prove that any isometric immersion of $\LO^2$ into
$\H_1^3$ admits a global parametrization by asymptotic curves. The
resulting coordinates are not \emph{Tchebysheff coordinates} in the
Euclidean sense, since the asymptotic curves in this Lorentzian
context cannot be parametrized by arc-length (indeed, they have
varying causal character). This detail is one of the main sources of
complication of the paper.

In Section 4 we improve the classical Dajczer-Nomizu theorem in
\cite{Dajczer} on the construction of timelike flat surfaces in
$\H_1^3$ as a product of two curves. More specifically, we use the
asymptotic coordinates constructed in Section 3 to prove that every
isometric immersion of $\LO^2$ into $\H_1^3$ can be recovered as the
pseudo-quaternionic product of two regular curves in $\H_1^3$, both in general with
varying causal character, that verify some compatibility conditions.

In Section 5 we show a general method to construct regular curves in
$\H_1^3$ that verify the hypotheses required by the classification
theorem of Section 4. This method is an extension to the Lorentzian
setting of Kitagawa's theory for studying complete flat surfaces in
$\S^3$. Here, we use the Hopf fibration of $\H_1^3$ over $\H^2$ and
we prove that such regular curves in $\H_1^3$ can be obtained as
\emph{asymptotic lifts} of curves with wave-front singularities in
$\H^2$. With this, we obtain our main result (Theorem \ref{clas1}), which
parametrizes the space of isometric immersions of $\LO^2$ into
$\H_1^3$ in terms of the space of curves with wave-front
singularities in $\H^2$.

Also in Section 5, we apply an idea by Kitagawa \cite{Kit1} and
Dadok-Sha \cite{DaSh} to prove that all Lorentzian flat tori of
$\H_1^3$ are exactly obtained when in Theorem \ref{clas1} one starts with
closed curves in $\H^2$, possibly with wave-front singularities, but
with a well defined unit normal at every point. Again, this provides
a parametrization of the space of Lorentzian flat tori in $\H_1^3$. We conclude this section analyzing in detail the example of Lorentzian Hopf cylinders and Lorentzian Hopf tori.

Finally, in Section 6 we give an answer to the Dajczer-Nomizu open
questions regarding the construction of isometric immersions of
$\LO^2$ into $\H_1^3$.

This work is part of the PhD Thesis of the first author.

\section{The geometry of $\H_1^3$}

Let $\R^4_2$ be the vector space $\R^4$ endowed with the semi-Riemannian metric
    \begin{equation*}
    \langle\hspace{6pt},\hspace{3pt}\rangle = - dx_0^2 - dx_1^2 + dx_2^2 +
    dx_3^2.
    \end{equation*}
The hypersurface $\mathbb{H}^3_1 = \{ x \in \R^4_2 : \langle x,
x\rangle = -1\}$ is then a model for the anti-de Sitter space of
dimension $3$. In this way, the induced metric on $\H^3_1$ is a Lorentzian metric of constant curvature $-1$. The space $\H^3_1$ is topologically a cylinder. Moreover, it is an $\S^1$-fibration over the hyperbolic plane $\H^2$ with timelike fibers and its universal covering $\widetilde{\H^3_1}$ is the unique Lorentzian space-form of constant curvature $-1$.

 Following the construction of \cite{amarillo}, we
will identify $\R^4_2$ with a certain set of maps
$\R^4_2\longrightarrow\R^4_2$, and $\mathbb{H}^3_1$ with a subset
of it. The composition induces a natural product structure on
$\R^4_2$ and $\mathbb{H}^3_1$, that will be seen then as Lie
groups.

Let consider $1 = {\rm Id}_{\R^4_2},$ and $i, j,
k:\R^4_2\longrightarrow\R^4_2$ given by:
        \begin{equation*}
        \begin{array}{rcl}
        i(x_0, x_1, x_2, x_3) &=& (x_1, -x_0,x_3,-x_2),\\
        j(x_0, x_1, x_2, x_3) &=& (x_2, -x_3, x_0, -x_1),\\
        k(x_0, x_1, x_2, x_3) &=& (x_3, x_2, x_1, x_0).
        \end{array}
        \end{equation*}
These maps verify:
    \begin{equation*}
    \begin{array}{rcccl}
    i^2 &=& i\circ i &=& -1,\\
    ij &=& j\circ i &=& k,\\
    ik &=& k\circ i &=& -j,
    \end{array}\quad
    \begin{array}{rcccl}
    ji &=& i\circ j &=& -k,\\
    j^2 &=& j\circ j &=& 1,\\
    jk &=& k\circ j &=& - i,
    \end{array}\quad
    \begin{array}{rcccl}
    ki &=& i\circ k &=& j,\\
    kj &=& j\circ k &=& i,\\
    k^2 &=& k\circ k &=& 1.
    \end{array}
    \end{equation*}
Note that we are using the letters $i$, $j$, $k$, as it is
usual for quaternions, but here the product structure is a
different one.

We consider now the vector space $\mathcal{F} = \text{span}\{1, i,
j, k\}$, and the isomorphism $\varphi: \mathcal{F}\longrightarrow
\R^4_2$ defined by $$\varphi(1) = \frac{\parc}{\parc x_0},
\hspace{0.3cm} \varphi(i) = \frac{\parc}{\parc x_1},
\hspace{0.3cm} \varphi(j) = \frac{\parc}{\parc x_2} \hspace{0.3cm}
\text{ and } \varphi(k) = \frac{\parc}{\parc x_3}.$$
In this way, $\R^4_2$ can be identified with the Lie group $\mathcal{F} = \{a + bi + cj +
dk:a,b,c,d\in\R\}$ endowed with the semi-Riemannian metric $\varphi^*\left(\langle\hspace{6pt},\hspace{3pt}\rangle\right)$ and we
will denote its metric simply by
$\langle\hspace{6pt},\hspace{3pt}\rangle$.

For $z = a + bi + cj + dk$ we use the notation ${\mathrm
Re}(z) = a$. We say that $z$ is \emph{real} (resp. \emph{pure
imaginary}) if $b=c=d=0$ (resp. $a=0$). Finally, we define the
\emph{conjugate} of $z$ as $\overline{z} = a - bi - cj - dk$.
It is easy to check that, given $z_1, z_2\in \R^4_2$,
$\overline{z_1 z_2} = \overline{z}_2
\hspace*{1pt}\overline{z}_1$. One can easily prove:

    \begin{pro} The following properties hold:
    \begin{itemize}
    \item[i)] For $z\in\R^4_2$, $\langle z, z\rangle
    = -z\overline{z} = -\overline{z}z=\langle \overline{z}, \overline{z}\rangle$.
    \item[ii)] In general, for $z_1, z_2 \in \R^4_2$, $\langle z_1, z_2\rangle
    = -{\mathrm Re}(z_1\overline{z}_2)$.
    \item[iii)] $z\in\mathbb{H}^3_1$ if, and only if, $z^{-1} =
    \overline{z}$.
    \item[iv)] $\langle\hspace{6pt},\hspace{3pt}\rangle$ is
    bi-invariant under multiplication by elements of $\mathbb{H}^3_1$, i. e., if $z_1, z_2\in \H^3_1$, then $\langle z_1 \,\eta \,z_2 , z_1 \,\rho\, z_2\rangle = \langle \eta, \rho\rangle$ .
    \end{itemize}
    \end{pro}

\vspace{3mm} Property $iv)$ tells that the Lie group structure induced on $\H^3_1$ by this quaternion-like product is its canonical Lie group structure, that is, the one for which its metric is bi-invariant. Besides, we have the identities:
    \begin{equation*}
    \H^3_1 = \{z\in\R^4_2: \langle z, z\rangle = -1\} =
            \{z\in\R^4_2: z\overline{z} = 1\} =
            \{z\in\R^4_2: \overline{z} = z^{-1}\}
    \end{equation*}

Observe also that that $1\in\mathbb{H}^3_1$ and that the vectors $\{i,
j, k\}$ form an orthonormal basis of $T_1\mathbb{H}^3_1$, i.e. $$\langle
i, i \rangle = -1, \langle j, j\rangle = \langle k, k\rangle =
1, \hspace{0.5cm} \esiz i,k\esde=\esiz j,k\esde=\esiz i,j\esde=0.$$ This basis can be extended to a global left-invariant orthonormal frame $\{E_1, E_2, E_3\}$ on $\H^3_1$ as:
     \begin{equation*}
    E_1(z) = z i,\qquad E_2(z) = z j,\qquad
    E_3(z) = z k\qquad\forall z\in\H^3_1.
    \end{equation*}

Taking into account that we are thinking of $\mathbb{H}^3_1$ as an
hypersurface of $\R^4_2$, there is a natural way to define a cross
product on each tangent space $T_{z}\mathbb{H}^3_1$. For $u,
v\in T_{z}\mathbb{H}^3_1$, $u\times v$ is the unique vector
in $T_{z}\mathbb{H}^3_1$ such that:
    \begin{equation*}
    \langle u\times v, w\rangle = \det(z, u, v, w)
    \quad\forall w \in T_{z}\mathbb{H}^3_1 .
    \]
In particular, $i\times j = k$, $j\times k =i$ and $k \times i = j$.

For a curve $a:I\longrightarrow\mathbb{H}^3_1$ with $a(0) = 1$,
and a vector field $X$ along $a$, we say that $X$ is \emph{left (resp.
right) invariant} along $a$ if, for all $t\in I$, $X(t) = a(t)
X(0)$ (resp. $X(t) = X(0) a(t)$). Let $\nabla$ denote the Levi-Civita connection of $\H_1^3$. The next Lemma is similar to the analogous result in the sphere $\S^3$ (see \cite{Kit1}, \cite{Spi}). Hence, we will omit the proof.

\begin{lem}\label{lemados} Let $a:I\longrightarrow\mathbb{H}^3_1$ be a curve with $a(0) =
1$ and $X$ a vector field along $a$. Then:
    \begin{itemize}
        \item[i)] $X$ is left invariant along $a$ if, and only if,
        $\nabla_{a'}X = a'\times X$.

        \item[ii)] $X$ is right invariant along $a$ if, and only if,
        $\nabla_{a'}X = X\times a'$.
    \end{itemize}
\end{lem}

To close this section, let us define now the family of \emph{Hopf fibrations} on
$\mathbb{H}^3_1$. For each nonzero pure imaginary $\rho\in\R^4_2$, we define the map $h_\rho:\H_1^3\flecha \R_2^4$ as
    \begin{equation}\label{quatrho}
    h_\rho(z) = z \hspace*{1pt}\rho \hspace*{1pt}\overline{z}
    \qquad \forall\hspace*{2pt}z\in \mathbb{H}^3_1
    \end{equation}

    \begin{pro}\label{Hopf_fibration}
    For every nonzero pure imaginary $\rho, \eta\in \R^4_2$ and every $z\in\mathbb{H}^3_1$
    we have:
        \begin{itemize}
        \item[i)] $\langle h_\rho(z), 1\rangle = 0$
        \item[ii)] $\langle h_\rho(z), h_\eta(z)
        \rangle = \langle \rho, \eta\rangle$. In particular, $\langle h_\rho(z),
        h_\rho(z)\rangle = \langle \rho, \rho\rangle$.
        \item[iii)] If $\langle \rho, \rho\rangle\leq 0$, then $\langle
        \rho, i\rangle$ and $\langle h_\rho(z),
        i\rangle$ have the same sign.
        \end{itemize}
    \end{pro}
\vspace{3mm}\noindent\textit{Proof:} \textit{i)} and \textit{ii)}
are consequence of the bi-invariance of the metric.

\noindent To prove \textit{iii)} we set $\varphi(z)=\langle
h_\rho(z), i\rangle$. Obviously, $\varphi$ is a continuous
function over $\mathbb{H}^3_1$ with $\varphi(1)=\langle\rho,
i\rangle$. If $\varphi$ changed sign, there would exist some $z_0\in\mathbb{H}^3_1$ such
that $\varphi(z_0) = 0$. But this is impossible because
$\varphi(z_0) = 0$ means that $h_\rho(z_0)$ has no part
on $i$ and, by \textit{i)} and \textit{ii)} we know that
$h_\rho(z_0)$ is pure imaginary with $\langle
h_\rho(z_0), h_\rho(z_0)\rangle = \langle \rho,
\rho\rangle \leq 0$. \hfill$\Box$

\vspace*{5mm} After Proposition~\ref{Hopf_fibration} we can
distinguish three fundamental types of maps $h_\rho$ by looking at
their images.
    \begin{equation*}
    \begin{array}{ll}
    h_+:\mathbb{H}^3_1 \longrightarrow \mathbb{S}^2_1(r)
    &\text{ if }\langle\rho, \rho \rangle = r^2,\vspace*{1mm}\\
    h_-:\mathbb{H}^3_1 \longrightarrow \left(\mathbb{H}^2(r)\right)^\pm
    &\text{ if }\langle\rho, \rho \rangle = - r^2 \hspace*{2pt}\text{ and }
    \hspace*{2pt}\langle\rho, i \rangle \lessgtr 0,\vspace*{1mm}\\
    h_0:\mathbb{H}^3_1 \longrightarrow \left(\Lambda^2\right)^\pm
    &\text{ if }\langle\rho, \rho \rangle = 0 \hspace*{2pt}\text{ and }
    \hspace*{2pt}\langle\rho, i \rangle \lessgtr 0.\\
    \end{array}
    \end{equation*}
Here,
    \begin{equation*}
    \begin{aligned}
    \mathbb{S}^2_1(r) &= \{z\in\R^4_2: \langle z, 1 \rangle = 0,
    \langle z, z \rangle = r^2\},\vspace*{1mm}\\
    \left(\mathbb{H}^2(r)\right)^\pm &=
    \{z\in\R^4_2: \langle z, 1 \rangle = 0, \langle z, z \rangle =
    -r^2, \langle z, i\rangle \lessgtr 0\},\vspace*{1mm}\\
    \left(\Lambda^2\right)^\pm &=
    \{z\in\R^4_2: \langle z, 1 \rangle = 0, \langle z, z \rangle =
    0, \langle z, i\rangle \lessgtr 0\},
    \end{aligned}
    \end{equation*}
i.e. $\left(\mathbb{H}^2(r)\right)^+$, $\left(\mathbb{H}^2(r)\right)^-$, $\left(\Lambda^2\right)^+$ and $\left(\Lambda^2\right)^-$ denote each of the connected component of $\H^2(r)$ and $\Lambda^2\backslash\{0\}$ respectively.

All the maps $h_\rho$ are fibrations over their corresponding base
manifolds and, since their definition is similar to that of the classical Hopf fibration $h:\S^3\longrightarrow \S^2$,
we will also call them \emph{Hopf fibrations}.

Now, we are going to focus on the fibrations $h_\rho$
with $\langle \rho, \rho \rangle = 1$, $\langle \rho, \rho \rangle
= -1$ or $\langle \rho, \rho \rangle = 0$. In those cases we will
denote simply by $\mathbb{S}^2_1$, $\left(\mathbb{H}^2\right)^\pm$
or $\left(\Lambda^2\right)^\pm$ their base manifold. Moreover, when no confusion can arise, we will also omit the reference to the connected component, using simply $\H^2$ or $\Lambda^2$. It is not
difficult to show that
    \begin{equation*}
    h_\rho (z_1) = h_\rho(z_2) \quad \Longleftrightarrow\quad
    z_2 = \pm \, z_1 e^{t\rho}
    \end{equation*}
where
    \begin{equation}\label{fibras}
    \begin{array}{ll}
    e^{t\rho} := \cosh({t})1 + \sinh(t)\rho\quad
    &\text{if }\hspace*{1pt} \langle \rho,\rho\rangle = 1,\\
    e^{t\rho} := \cos({t})1 + \sin(t)\rho\quad
    &\text{if }\hspace*{1pt} \langle \rho,\rho\rangle = -1,\\
    e^{t\rho} := 1 + t\rho\quad
    &\text{if }\hspace*{1pt} \langle \rho,\rho\rangle = 0.
    \end{array}
    \end{equation}

\section{Isometric immersions of $\LO^2$ into $\H_1^3$}

Consider an isometric immersion $f:\LO^2\flecha \H_1^3$ from the
Lorentz plane $\LO^2$ into the anti-de Sitter $3$-space $\H_1^3$.
Here, $\LO^2$ will be viewed as the vector space $\R^2$ endowed
with the Lorentzian metric $ds^2 =-dx^2+dy^2$ in canonical
coordinates $(x,y)$. Before starting, let us remark that most of
what follows can be adapted for (not necessarily complete) simply
connected Lorentzian flat surfaces in $\H_1^3$, see the remark at
the end of this section.

Let $N(x,y):\R^2\flecha \S_2^3 :=\{x\in \R_2^4 : \esiz x,x\esde
=1\}$ denote the unit normal of the immersion $f$, chosen so that
the frame \[ \{f,f_x,f_y,N  \}
\]
is a positively oriented orthonormal frame in the manifold
$\R^4_2$. Then, the first and second fundamental forms of the
immersion are given, respectively, by

\begin{equation}\label{1f2f}
\left\{\def\arraystretch{1.3} \begin{array}{lll}
I & = & \langle df, df \rangle= -dx^2+dy^2, \\
II &= & -\langle df, dN \rangle= adx^2+2bdxdy+cdy^2,
\end{array}\right.
\end{equation}
where $a:=-\langle f_x, N_x \rangle, b:=- \langle f_x, N_y
\rangle$ and  $c:=-\langle f_y, N_y \rangle$ satisfy the
Gauss-Codazzi equations $$a_y = b_x, \hspace{0.6cm} c_x=b_y,
\hspace{0.6cm} ac -b^2 =-1.$$ Thus, there is some $\phi(x,y)\in
C^{\8} (\R^2)$ which is a solution to the hyperbolic Monge-Ampère
equation $$\phi_{xx} \phi_{yy} - \phi_{xy}^2 =-1$$ such that
$a=\phi_{xx}$, $b= \phi_{xy}$ and $c=\phi_{yy}$. Hence,

\begin{equation}\label{2phi}
II= \phi_{xx} dx^2 + 2\phi_{xy} dx dy + \phi_{yy} dy^2,
\hspace{1cm} \phi_{xx} \phi_{yy} - \phi_{xy}^2 =-1.
\end{equation}
This implies that, associated to $f$, there exists an Euclidean
isometric immersion $\tilde{f}(x,y):\R^2\flecha \S^3$ of the
Euclidean plane into the unit $3$-sphere $\S^3$ with first and
second fundamental forms given, respectively, by

\begin{equation}\label{12e}
\left\{\def\arraystretch{1.3} \begin{array}{lll}
\tilde{I} & = & dx^2+dy^2, \\
II &= & \phi_{xx} dx^2 + 2\phi_{xy} dx dy + \phi_{yy} dy^2.
\end{array}\right.
\end{equation}
This is just a consequence of the classical fact that
$(\tilde{I},II)$ as in \eqref{12e} verify the Gauss-Codazzi
equations for surfaces in $\S^3$. This correspondence was observed
with a different formulation by Dajczer and Nomizu \cite{Dajczer}. It must be emphasized that this correspondence is not \emph{geometric}, in the sense that it depends on the specific coordinates $(x,y)$ in $\LO^2$ that we choose. In other words, two different global Lorentzian coordinates $(x,y)$ and $(x',y')$ in $\LO^2$ differing by an isometry generate, in general, two non-congruent flat surfaces in $\S^3$.

Now, since $\tilde{f}$ is a complete flat surface in $\S^3$, it is
classically known (see \cite{Spi} for instance) that there exist
globally defined \emph{Tchebysheff coordinates} $(u,v)$ on the
surface. In other words, we may parametrize the surface as
$\tilde{f} (u,v):\R^2\flecha \S^3$ so that

\begin{equation}\label{te}
\left\{\def\arraystretch{1.3} \begin{array}{lll}
\tilde{I}& =& du^2+2\text{cos} \omega dudv +dv^2, \\
II&=&2 \text{sin}  \omega   dudv,
\end{array}\right.
\end{equation}
where $\omega (u,v)\in C^{\8} (\R^2)$ satisfies $0<\omega
(u,v)<\pi$ and $\omega_{uv}=0$. Note that from the expression of $II$ in \eqref{te} it is clear that the $u$-curves and the $v$-curves are the asymptotic curves of the immersion $\tilde{f}$.

Let us now find the explicit formula of the global diffeomorphism
of $\R^2$ given by the change of coordinates $$(u,v)\mapsto
(x(u,v), y(u,v)).$$ By comparing $\tilde{I}$ in \eqref{12e} and
\eqref{te} we get

\begin{equation}\label{sispa}
\left\{\begin{array}{ccl}
  x_u^2+y_u^2 & = & 1, \\
  x_ux_v+y_uy_v & = & \text{cos} \, \omega (u,v), \\
  x_v^2+y_v^2 & = & 1.
\end{array}\right.
\end{equation}
Any solution to \eqref{sispa} must be of the form
 \begin{equation}\label{sispa2}
 \def\arraystretch{1.2}\begin{array}{cc}
  x_u=\text{cos}\omega_1, & y_u=\text{sin}\omega_1, \\
  x_v=\text{cos}\omega_2, & y_v=-\text{sin}\omega_2.
\end{array}
 \end{equation}
where $\omega_i\in C^{\8}(\R^2)$ satisfy $\omega_1 +\omega_2 =
\omega $ (these functions are uniquely determined up to changes of the form
$\omega_1 \mapsto \omega_1 + 2 k\pi$, $\omega_2 \mapsto \omega_2 - 2 k\pi$, with $k\in\Z$). Using now that $(x_u)_v =(x_v)_u$
and $(y_u)_v= (y_v)_u$ we get
\[
\def\arraystretch{1.2}\begin{array}{llr}
  -(\omega_1)_v \sin \omega_1  &= & -(\omega_2)_u \sin \omega_2 \\
   (\omega_1)_v \cos \omega_1  &= & -(\omega_2)_u \cos \omega_2
\end{array}
\] i.e. either $(\omega_1)_v = (\omega_2)_u =0$ or $\sin (\omega_1+ \omega_2)=0$, the latter being not possible since $\sin \omega (u,v) \in (0,\pi)$. Thus, the function $\omega(u,v)$ appearing in \eqref{te} can be put in the form

 \begin{equation}\label{wave}
 \omega (u,v)= \omega_1 (u) + \omega_2 (v), \hspace{1cm} \omega_i \in C^{\8} (\R).
 \end{equation}
From here, the coordinates $(x,y)$ are given in terms of $(u,v)$ by
\begin{equation}\label{camcor}
\left\{\def\arraystretch{1.2}\begin{array}{lll}
  x(u,v)=\int \text{cos}\omega_1 du + \int \text{cos}\omega_2 dv +c_1,\\
  y(u,v)=\int \text{sin}\omega_1 du - \int \text{sin}\omega_2 dv
  +c_2,
\end{array}\right.
\end{equation}
where $c_1$ and $c_2$ are integration constants that can be chosen
to be zero, up to a translation in the $(x,y)$ plane. In
particular, the map given by \eqref{camcor} is a global
diffeomorphism of $\R^2$, whenever we start with a complete flat
surface in $\S^3$.

\begin{remark}\label{remcom}
Let us point out that the map $(x(u,v),y(u,v)):\R^2\flecha \R^2$ given by \eqref{camcor} is a global diffeomorphism if and only if the Riemannian metric 
 \begin{equation}\label{eccco}
 \tilde{I} = du^2 + 2\cos (\omega_1(u) + \omega_2(v)) \, du dv + dv^2
  \end{equation}
is complete.
\end{remark}

Once here, we can use \eqref{camcor} to express the Lorentzian
metric $I=-dx^2+dy^2$ in terms of the global
$(u,v)$-coordinates. First of all, let us observe that

\begin{equation*}
\begin{array}{lll}
 dx^2  &= &x_u^2du^2+2 x_ux_vdudv+x_v^2dv^2  \\
   &= &\text{cos}^2 \omega_1 du^2+ 2 \cos\omega_1
  \cos \omega_2dudv+ \text{cos}^2 \omega_2dv^2.
\end{array}
\end{equation*}
And then, the Lorentzian metric can be expressed as
\begin{equation*}
    \begin{aligned}
    -dx^2+dy^2 &= -2dx^2+dx^2+dy^2 \\
    &= -2dx^2 + du^2+2 \cos \omega dudv +dv^2\\
    &= - \text{cos}(2\omega_1 )du^2 - 2 \cos (\omega_1 -
\omega_2)dudv -\text{cos}(2\omega_2 )dv^2.
    \end{aligned}
    \end{equation*}

We have from the above discussion the following result.

\begin{pro}\label{cara}
Let $\omega_1 (u),\omega_2(v)\in C^{\8}(\R)$ such that
$\omega_1(u)+\omega_2(v)\in (0,\pi)$ for all $(u,v)\in \R^2$.
Then, there exists an immersion $f(u,v):\R^2\flecha \H_1^3$ whose
first, second and third fundamental forms are given by

\begin{equation}\label{tel}
\left\{\def\arraystretch{1.5} \begin{array}{lrrrr} I& =& -
\cos(2\omega_1 )du^2 &-  2 \cos (\omega_1 -
\omega_2)dudv &-\cos(2\omega_2 )dv^2, \\
II&=& &2\sin (\omega_1 + \omega_2) dudv  & \\
III & = & \cos (2\omega_1 )du^2 &-  2 \cos (\omega_1 -
\omega_2)dudv  &+ \cos(2\omega_2 )dv^2
\end{array}\right.
\end{equation}
In this way, $f$ describes a flat timelike surface in $\H_1^3$ whose
asymptotic curves are the images of the coordinate curves in the
$(u,v)$-plane. Moreover, $f$ represents an isometric immersion of
$\LO^2$ into $\H_1^3$ exactly when the local diffeomorphism
$(x(u,v),y(u,v))$ of $\R^2$ given by \eqref{camcor} is actually a
global diffeomorphism. A sufficient condition for \eqref{camcor}
to be a global diffeomorphism is that

\begin{equation}\label{st0}
0<c_1 \leq \omega_1 (u) + \omega_2 (v) \leq c_2 <\pi \hspace{1cm}
\forall (u,v)\in \R^2.
\end{equation}

Conversely, any isometric immersion $f(x,y):\LO^2\flecha \H_1^3$
admits a parametrization $f(u,v):\R^2\flecha \H_1^3$ such that
\eqref{tel} holds for some $\omega_1 (u),\omega_2(v)\in
C^{\8}(\R)$ verifying $\omega_1(u)+\omega_2(v)\in (0,\pi)$ for all
$(u,v)\in \R^2$. In that situation, the change of coordinates
$(u,v)\mapsto (x(u,v),y(u,v))$ is given by \eqref{camcor}.

\end{pro}
 \begin{proof}
All the statements of the converse part follow from the previous
discussion, except for the expression of the third fundamental
form $III=\esiz dN,dN\esde$. In this sense, a standard derivation
of the Gauss-Weingarten formulas of the immersion
$f(u,v):\R^2\flecha \H_1^3$ yields

\begin{equation}\label{gausfor} \left\{ \def\arraystretch{1.9} \begin{array}{lll}

    f_{uu} &=& \dfrac{\omega_1' \cos(\omega_1 + \omega_2)}{\sin(\omega_1 + \omega_2)} f_u
    - \dfrac{\omega_1'}{\sin(\omega_1 + \omega_2)} f_v
    - \cos(2\omega_1) f,\\
    f_{uv} &=& \sin(\omega_1 + \omega_2)) N
    - \cos(\omega_1 - \omega_2)f, \\
    f_{vv} &=& -\dfrac{\omega_2'}{\sin(\omega_1 + \omega_2)} f_u
    + \dfrac{\omega_2' \cos(\omega_1 + \omega_2)}{\sin(\omega_1 + \omega_2)} f_v
    - \cos(2\omega_2) f
    \end{array} \right.
    \end{equation}
and
\begin{equation}\label{weinfor} \left\{ \def\arraystretch{2} \begin{array}{lll}

     N_u &=& \dfrac{\cos(\omega_1 - \omega_2)}{\sin(\omega_1 + \omega_2)} f_u
    - \dfrac{\cos(2\omega_1)}{\sin(\omega_1 + \omega_2)} f_v,
   \\
    N_v &=& -\dfrac{\cos(2\omega_2)}{\sin(\omega_1 + \omega_2)} f_u
    + \dfrac{\cos(\omega_1 - \omega_2)}{\sin(\omega_1 + \omega_2)}f_v.
    \end{array} \right.
    \end{equation}
From \eqref{weinfor} and the expression of $I$ in \eqref{tel} we
get that $$III  =  \esiz dN,dN\esde = \cos (2\omega_1 )du^2 -  2
\cos (\omega_1 - \omega_2)dudv  + \cos(2\omega_2 )dv^2,$$ as
wished.

Now, assume that we are given $\omega_1(u),\omega_2 (v)\in C^{\8}
(\R)$ with $\omega_1 +\omega_2 \in (0,\pi)$. Then, the existence
of the immersion $f(u,v):\R^2\flecha \H_1^3$ such that \eqref{tel}
holds follows from the Gauss-Codazzi equations and
\eqref{gausfor}, \eqref{weinfor}. The metric $I$ is flat and
timelike, and if we use the local coordinates $(x,y)$ given by
\eqref{camcor}, we have
 \begin{equation}\label{strr}
 I= -dx^2 + dy^2.
 \end{equation}
So, clearly, $f$ will describe an isometric immersion of $\LO^2$
into $\H_1^3$ with canonical coordinates $(x,y)$ if \eqref{camcor}
is a global diffeomorphism. Conversely, assume that $f$ describes
an isometric immersion of $\LO^2$ into $\H_1^3$ with canonical
coordinates $(x',y')$. Then $I= -dx^2 + dy^2 = -dx'^2 + dy'^2$ by
\eqref{strr}. But this implies that $(x,y)$ and $(x',y')$ differ
by an isometry of $\LO^2$, and hence \eqref{camcor} is a global
diffeomorphism.

Finally, if \eqref{st0} holds and we denote $\Phi (u,v)=
(x(u,v),y(u,v)):\R^2\flecha \R^2$, then we get immediately from
\eqref{camcor} that the gradient of $\Phi^{-1}$ has bounded norm
around any point, i.e. $$ || D (\Phi^{-1}) || \leq M < \8$$ for
some $M>0$. So, $\Phi$ is a global diffeomorphism by the
Hadamard-Plastock inversion theorem. (Alternatively, if \eqref{st0} holds, then $|\cos \omega|\leq c_0<1$ for some $c_0$, and so $\tilde{I} \geq (1-c_0^2) (du^2+dv^2)$ for the Riemannian metric in \eqref{eccco}; thus $\tilde{I}$ is complete, and by Remark \ref{remcom}, $\Phi$ is a global diffeomorphism).
 \end{proof}

\begin{remark}
\emph{In the previous arguments, the only place
where completeness plays a role is in the existence of the
global parameters $(u,v)$. Nonetheless, these
parameters always exist locally, as can be deduced from
\eqref{camcor} and the fact that the \emph{flat} coordinates
$(x,y)$ always exist locally for any (abstract) Lorentzian flat
surface. Moreover, if we start with a simply connected Lorentzian
flat surface $\Sigma$, then one can still choose a
\emph{coordinate immersion} $(x,y):\Sigma\flecha \LO^2$ into the
Lorentz plane that serves as a substitute to the one-to-one
coordinates $(x,y)$ that exist locally or for complete Lorentzian
flat surfaces (see \cite{AGM1}).}

\emph{It comes then clear from these comments that all the previous
process can be readily formulated for arbitrary simply connected
Lorentzian flat surfaces isometrically immersed in $\H_1^3$. Obviously, in that case we should not impose that \eqref{camcor} is a global diffeomorphism.}
\end{remark}
\begin{remark}
\emph{In the Euclidean case, the functions
$\omega_i$ in \eqref{wave} are uniquely determined up to the
change}
 \begin{equation}\label{ambig}
 \omega_1(u)\mapsto \omega_1(u) + c, \hspace{1cm} \omega_2(v) \mapsto \omega_2(v) -c ,\hspace{1cm} c\in \R.
 \end{equation}
\emph{In the present Lorentzian case, this ambiguity does not hold anymore, i.e. a change like \eqref{ambig} also changes the resulting Lorentzian flat surface in $\H_1^3$.}
\end{remark}

\section{A representation formula}

Our aim in this section is to prove that any isometric immersion
of $\LO^2$ into $\H_1^3$ can be represented, with respect to the
characteristic parameters $(u,v)$ provided by Proposition
\ref{cara}, as the product of two adequate curves in $\H_1^3$. We
split this result into two separate theorems.

The first one is:

 \begin{teo}\label{representation1}
    Let $f(u,v):\R^2\longrightarrow\mathbb{H}^3_1$ be an isometric
    immersion of $\mathbb{L}^2$ into $\mathbb{H}^3_1$ where
    $(u,v)$ are the global characteristic parameters given in Proposition \ref{cara}.
    Let $N(u,v):\R^2\longrightarrow\mathbb{S}^3_2$ denote its unit normal, and assume without loss of generality that $f(0,0) =
    1$ and $N(0,0)=j$. Then, we have
        \begin{equation}\label{es3}
        f(u,v) = a_1(u)a_2(v),\qquad N(u,v) = a_1(u)ja_2(v),
        \end{equation}
    for $a_1(u) := f(u,0)$ and $a_2(v) := f(0,v)$. Moreover,
    these two asymptotic curves verify
       \begin{equation}\label{dobless}
       \langle a_1'(u), a_1(u) j\rangle = 0 = \langle a_2'(v), j
        a_2(v)\rangle.
       \end{equation}
    \end{teo}
To  prove Theorem \ref{representation1}, we will use the following
result.
    \begin{lem}\label{lemaus}
    Under the hypotheses of Theorem \ref{representation1}, we have:
        \begin{itemize}
        \item[i)] $N$, $f_v$ and $N_v$ are left invariant along $a_1(u)$.
        \item[ii)] $N$, $f_u$ and $N_u$ are right invariant along $a_2(v)$.
        \end{itemize}

    \end{lem}

 \begin{proof}
By \eqref{tel} we see that $N_u$ is
orthogonal to $N$, $f$ and $f_u$ and that $\langle N_u, N_u
\rangle = \cos (2\omega_1) = - \langle f_u, f_u \rangle$.
Therefore, we have $N_u = \pm f_u\times N$.

To determine this sign we take into account that, since $\omega_1
+ \omega_2 \in (0,\pi)$,
    \begin{equation*}
    \begin{aligned}
    0 &> - \sin(\omega_1 + \omega_2) = \langle N_u, f_v\rangle
    = \langle \pm f_u \times N, f_v\rangle\\
    &= \mp \langle
    f_u\times f_v, N\rangle
    = \mp\left\langle f_u\times f_v , \dfrac{f_u\times f_v}{||f_u\times
    f_v||}\right\rangle\\ &= \mp ||f_u\times f_v||.
    \end{aligned}
    \end{equation*}
Then, we deduce that
    \begin{equation}\label{norma f_u x f_v}
    ||f_u\times f_v|| = \sin(\omega_1 + \omega_2)
    \end{equation}
and that
    \begin{equation*}
    N_u(u,v) = f_u (u,v) \times N(u,v).
    \end{equation*}
In particular,
    \begin{equation*}
    \nabla_{a_1'} N = N_u(u,0) = f_u(u,0) \times N(u,0) = a_1'\times
    N.
    \end{equation*}
So, by Lemma \ref{lemados} we conclude that $N$ is left invariant
along $a_1$. A similar argument yields
    \begin{equation}\label{expresion N_v}
    N_v(u,v) = N(u,v) \times f_v(u,v).
    \end{equation}
Thus, particularizing at points of the form $(0,v)$ we can apply
Lemma \ref{lemados} to deduce that $N$ is right invariant along
$a_2$.

Now, we consider the vector field $f_{uv}(u,v)$. Using
 \eqref{gausfor} and \eqref{norma f_u x f_v} we get
    \begin{equation*}
    f_{uv}(u,v) = f_u(u,v)\times f_v(u,v) - \cos(\omega_1 -
    \omega_2) f(u,v).
    \end{equation*}
At points of the form $(u,0)$, this equality provides
    \begin{equation*}
    \nabla_{a_1'}f_v = \big(f_{uv}(u,0)\big)^\top = f_u(u,0)\times
    f_v(u,0) = a_1'\times f_v.
    \end{equation*}
Again, Lemma \ref{lemados} gives the desired conclusion. The fact
that $f_u$ is right invariant along $a_2$ is obtained in the same
way.

Finally, if we use left invariancy along $a_1$ of $N$ and $f_v$ in
\eqref{expresion N_v}, we obtain
    \begin{equation*}
    \begin{aligned}
    N_v(u,0) &= N(u,0) \times f_v(u,0) =
    \big(a_1(u)N(0,0)\big)\times \big(a_1(u) f_v(0,0)\big)\\
    &= a_1(u) \big(N(0,0) \times f_v(0,0)\big) = a_1(u) N_v(0,0),
    \end{aligned}
    \end{equation*}
that is, $N_v$ is left invariant along $a_1$. Also in this case,
we can use similar arguments to prove that $N_u$ is right
invariant along $a_2$. This finishes the proof of Lemma
\ref{lemaus}.
 \end{proof}

 \vspace{5mm}\noindent\textit{Proof of
Theorem \ref{representation1}:}  First of all, let us observe that
\eqref{dobless} follows from $\esiz f_u,N\esde = \esiz f_v,N\esde
=0$ at points of the form $(u,0)$ or $(0,v)$, and from the
left-right invariance of $N$ given by Lemma \ref{lemaus}.

In order to prove \eqref{es3}, we start by combining the structure
equations \eqref{gausfor} and \eqref{weinfor} with basic
trigonometric laws to obtain
    \begin{equation}\label{Gauss1+Weingarten1}
    \omega_1'\big(N_u - \sin(2\omega_1) f_u\big) =
    \cos(2\omega_1)\big(f_{uu} + \cos(2\omega_1) f\big).
    \end{equation}
and 
\begin{equation}\label{45es}
N_u - \sin (2\omega_1) f_u = \frac{\cos (2\omega_1)}{\sin (\omega_1+\omega_2)} \left( \cos (\omega_1 +\omega_2) f_u - f_v.\right)
\end{equation}
Now, for a fixed $v_0$ we define the curves $\Gamma_1,
\Gamma_2:\R\longrightarrow\mathbb{H}^3_1$ as
    \begin{equation*}
    \Gamma_1 (u) = f(u,v_0),\qquad\Gamma_2(u) = a_1(u)a_2(v_0).
    \end{equation*}
Next, let us construct frame along each of these two curves. It is important to observe that $\Gamma_1$ and $\Gamma_2$ do not have constant causal character. Thus, this frame that we introduce here is not the Frenet frame of the curve, and has to be constructed \emph{ad hoc}. So, consider:
    \begin{equation*}
    \left\{
    \begin{array}{ccl}
    \vec{t}_1(u) &=& \Gamma_1'(u)= f_u(u,v_0),\vspace*{2mm}\\
    \vec{n}_1(u)&=&\dfrac{\cos(\omega_1(u) + \omega_2(v_0))}{\sin(\omega_1(u) + \omega_2(v_0))}f_u(u,v_0)
    - \dfrac{1}{\sin(\omega_1(u) +
    \omega_2(v_0))}f_v(u,v_0),\vspace*{2mm}\\
    \vec{b}_1(u) &=& N(u,v_0).
    \end{array}
    \right.
    \end{equation*}
    \begin{equation*}
    \left\{
    \begin{array}{ccl}
    \vec{t}_2(u) &=& \Gamma_2'(u)= a_1'(u)a_2(v_0),\vspace*{2mm}\\
    \vec{n}_2(u)&=&\dfrac{1}{\cos(2\omega_1(u))}\Big(a_1'(u)ja_2(v_0)
    - \sin(2\omega_1(u))a_1'(u)a_2(v_0)\Big),\vspace*{2mm}\hspace*{11mm}\\
    \vec{b}_2(u) &=& a_1(u) j a_2(v_0).
    \end{array}
    \right.
    \end{equation*}
Note that the definition of $\vec{n}_2$ is valid, at first, only
at points with $\cos(2\omega_1(u))\neq 0$. However, using the left
invariancy of $N$ along $a_1$ and \eqref{45es}, we get
    \begin{equation}\label{n_2 bien definido}
    \begin{aligned}
    &\vec{n}_2(u) = \dfrac{1}{\cos(2\omega_1(u))}\Big(a_1'(u)j
    - \sin(2\omega_1(u))a_1'(u)\Big)a_2(v_0)\\
    &=\dfrac{1}{\cos(2\omega_1(u))}\Big(N_u(u,0)
    - \sin(2\omega_1(u))f_u(u,0)\Big)a_2(v_0)\\
    &=\left(
    \dfrac{\cos(\omega_1(u) + \omega_2(0))}{\sin(\omega_1(u) + \omega_2(0))}f_u(u,0)
    - \dfrac{1}{\sin(\omega_1(u) +
    \omega_2(0))}f_v(u,0)\right)a_2(v_0).\\
    \end{aligned}
    \end{equation}
Thus, $\vec{n}_2 (u)$ is actually well defined for all $u$.

We claim now that the references
$\{\vec{t}_1,\vec{n}_1,\vec{b}_1\}$ and
$\{\vec{t}_2,\vec{n}_2,\vec{b}_2\}$ coincide at $u=0$. Indeed, for
$\vec{t}_i$ and $\vec{b}_i$, this is a direct consequence of the
fact that $f_u$ and $N$ are right invariant along $a_2$. For
$\vec{n}_i$, the result follows from the second identity in \eqref{n_2 bien definido} evaluated at $u=0$, and the right invariance of $N_u$ along $a_2$:

$$\def\arraystretch{2.1}\begin{array}{lll}
\vec{n}_2 (0)& = &\displaystyle\frac{1}{\cos (2\omega_1 (0))} \left(N_u (0,0) a_2 (v_0) - \sin (2\omega_1(0)) f_u (0,0) a_2 (v_0) \right) \\ 
& = &\displaystyle\frac{1}{\cos (2\omega_1 (0))} \left(N_u (0,v_0) - \sin (2\omega_1(0)) f_u (0,v_0) \right) \\ 
& = & \left(
    \dfrac{\cos(\omega_1(0) + \omega_2(v_0))}{\sin(\omega_1(0) + \omega_2(v_0))}f_u(0,v_0)
    - \dfrac{1}{\sin(\omega_1(0) +
    \omega_2(v_0))}f_v(0,v_0)\right) \\ & = & \vec{n}_1 (0).
\end{array}$$

Finally, we are going to show that both references verify the same system of
differential equations. In the case of
$\{\vec{t}_1,\vec{n}_1,\vec{b}_1\}$ we just have to apply
\eqref{gausfor} and \eqref{weinfor} to deduce that
    \begin{equation*}
    \nabla_{\Gamma_1'}\vec{t}_1 = \big(f_{uu}(u,v_0)\big)^\top= \omega_1'(u)\vec{n}_1(u).
    \end{equation*}

    \begin{equation}\label{exs}
    \begin{aligned}
    \nabla_{\Gamma_1'}\vec{n}_1 &= \dfrac{\partial}{\partial u}\left(\dfrac{\cos(\omega_1(u) + \omega_2(v_0))}{\sin(\omega_1(u) + \omega_2(v_0))}\right)f_u(u,v_0) + \dfrac{\cos(\omega_1(u) + \omega_2(v_0))}{\sin(\omega_1(u) + \omega_2(v_0))} \big(f_{uu}(u,v_0)\big)^\top\\
    &\hspace*{5mm} - \dfrac{\partial}{\partial u}\left(\dfrac{1}{\sin(\omega_1(u) + \omega_2(v_0))}\right)
    f_v(u,v_0) - \dfrac{1}{\sin(\omega_1(u) + \omega_2(v_0))} \big(f_{vu}(u,v_0)\big)^\top\\
    &= - \omega_1'(u) f_u(u,v_0) - N(u,v_0)\\
    &= - \omega_1'(u) \vec{t}_1(u) - \vec{b}_1(u).
    \end{aligned}
    \end{equation}

    \begin{equation*}
    \begin{aligned}
    \nabla_{\Gamma_1'}\vec{b}_1 &= \big(N_u(u,v_0)\big)^\top\\
    &= \dfrac{\cos(\omega_1(u) - \omega_2(v_0))}{\sin(\omega_1(u) + \omega_2(v_0))} f_u(u,v_0) - \dfrac{\cos(2\omega_1(u))}{\sin(\omega_1(u) + \omega_2(v_0))}f_v(u,v_0)\\
    &= \sin(2\omega_1(u)) f_u(u,v_0)\\
     &\hspace*{5mm} + \cos(2\omega_1(u))\left(\dfrac{\cos(\omega_1(u) + \omega_2(v_0))}{\sin(\omega_1(u) + \omega_2(v_0))}f_u(u,v_0)
    - \dfrac{1}{\sin(\omega_1(u) +
    \omega_2(v_0))}f_v(u,v_0)\right)\\
    &= \sin(2\omega_1(u))\vec{t}_1(u) + \cos(2\omega_1(u))\vec{n}_1(u).
    \end{aligned}
    \end{equation*}

    To obtain the differential equations of the reference
    $\{\vec{t}_2,\vec{n}_2,\vec{b}_2\}$, the main idea is to use
    that $a_1''(u) = f_{uu}(u,0)$\hspace*{1pt} whenever we find the terms
    $\big(a_1''(u)\big)^\top a_2(v_0)$ or $\big(a_1''(u)\big)^\top j a_2(v_0)$. Then, we can
    apply \eqref{Gauss1+Weingarten1} and express
    $\big(f_{uu}(u, 0)\big)^\top$ in terms of
    $N_u(u,0)$ and $f_u(u,0)$. The last step is to use that $f_u(u,0) = a_1'(u)$
    and $N_u(u,0) = a_1'(u)j$ (which follows from the left-invariancy of
    $N$ along $a_1$) and rewrite all terms as products of the curves $a_1$ and $a_2$
     and their derivatives.

Using the above scheme, we have

    \begin{equation*}
    \begin{aligned}
    \nabla_{\Gamma_2'}\vec{t}_2 &= \big(a_1''(u)\big)^\top a_2(v_0) =\big(f_{uu}(u,0)\big)^\top a_2(v_0)\\
    &=\left(\dfrac{\omega_1'(u)}{ \cos(2\omega_1(u))}
    N_u(u,0) - \dfrac{\omega_1'(u)\sin(2\omega_1(u))}{\cos(2\omega_1(u))} f_u(u,0)\right)
    a_2(v_0)\\
    &= \dfrac{\omega_1'(u)}{ \cos(2\omega_1(u))} \big(a_1'(u)ja_2(v_0)
    - \sin(2\omega_1(u)) a_1'(u)a_2(v_0)\big)\\
    &= \omega_1'(u)\vec{n}_2(u).
    \end{aligned}
    \end{equation*}
Next, using \eqref{n_2 bien definido} and performing basically the
same computation that in \eqref{exs}, we get
    \begin{equation*}
    \begin{aligned}
    \nabla_{\Gamma_2'}\vec{n}_2 & = - \omega_1'(u) a_1'(u) a_2(v_0) - a_1(u) j
    a_2(v_0)\\
    &= - \omega_1'(u) \vec{t}_2(u) - \vec{b}_2(u).
    \end{aligned}
    \end{equation*}
At last,
    \begin{equation*}
    \begin{aligned}
    \nabla_{\Gamma_2'}\vec{b}_2 &= a_1'(u)j a_2(v_0)\\
    &= \cos(2\omega_1(u))
    \left(\dfrac{1}{\cos(2\omega_1(u))}
    \big(a_1'(u)ja_2(v_0) -
    \sin(2\omega_1(u))a_1'(u)a_2(v_0)\big)\right)\\
    &\hspace*{5mm} + \sin(2\omega_1(u)) a_1'(u) a_2(v_0)\\
    &= \sin(2\omega_1(u))\vec{t}_2(u) + \cos(2\omega_1(u))\vec{n}_2(u).
    \end{aligned}
    \end{equation*}

Therefore, we have proved that $\{\vec{t}_1,\vec{n}_1,\vec{b}_1\}$
and $\{\vec{t}_2,\vec{n}_2,\vec{b}_2\}$ agree at $u=0$ and verify
the same system of differential equations. Hence, we can conclude that these
two references coincide along $\R$. In particular, we deduce that
$\Gamma_1 \equiv \Gamma_2$. Since this can be
done for any $v_0$, we obtain that $f(u,v) = a_1(u)a_2(v)$ and
also that $N(u,v) = a_1(u) j a_2(v)$. This concludes the proof of
Theorem \ref{representation1}.
\hfill$\Box$

\vspace{0.4cm}

Our objective now is to study the converse of Theorem
\ref{representation1}. In other words, we wish to obtain a method
to construct flat surfaces from two given curves in $\H_1^3$,
satisfying some conditions.

Let us first note that if a curve
$a:\mathbb{R}\rightarrow\mathbb{H}^3_1$ verifies $\langle a',
aj\rangle = 0$, then, for the curve $\bar{a}$ we have $\langle
\bar{a}', j\bar{a}\rangle = 0$. This provides a simplification of
condition \eqref{dobless}. Namely, after this observation, we are
left with the problem of finding out if two given curves $a_1(u),
a_2(v): \R\flecha \H_1^3$, both verifying $\esiz a_i',a_i j\esde =
0$, always describe an isometric immersion of $\LO^2$ into $\H_1^3$
such that $f(u,0)=a_1(u)$ and $f(0,v)= \overline{a_2(v)}$ for the
characteristic parameters given in Proposition \ref{cara}. In
order to do this, we introduce the following terminology.

Let $a:\R\longrightarrow\mathbb{H}^3_1$ be a regular curve such
that $\langle a'(s), a(s) j\rangle = 0$ for all $s\in \R$. Then,
we can write $\overline{a(s)}a'(s) = \lambda(s) i + \mu (s) k$ for $\landa,\mu \in C^{\8} (\R)$.

    \begin{defi}\label{parametro_asintotico_omega}
    In the above situation, we say that $s$ is the \emph{asymptotic parameter}
    of the curve $a$
    if $\lambda(s)^2 + \mu(s)^2 = 1$. In that case, we can write
        \begin{equation}\label{definicion_omega}
        \overline{a(s)}a'(s) = \cos(\omega^a(s))i +
        \sin(\omega^a(s))k
        \end{equation}
   for some $\omega^{a}\in\mathcal{C}^{\infty}(\R)$, which is
   uniquely determined up to translations of the form $\omega^{a}
   \mapsto \omega^{a} + 2k\pi$, with $k\in\Z$.
    \end{defi}

Obviously, any curve in $\H_1^3$ with $\esiz a',a j\esde =0$ can be re-parametrized by its asymptotic parameter. With this, the following result is a converse to Theorem
\ref{representation1}, and completes the desired representation
theorem.

    \begin{teo} \label{representation2}
    Let $a_1(u), a_2(v):\R\longrightarrow \mathbb{H}^3_1$ be two regular curves,
    with $a_1(0) = 1 = a_2(0)$, satisfying:
        \begin{itemize}
            \item[i)] $\langle a_i', a_i j\rangle = 0$, for $i =
            1, 2$.
            \item[ii)] $u$ and $v$ are the asymptotic parameters of $a_1$ and $a_2$, respectively.
            \item[iii)] The functions $\omega_ 1 = \omega^{a_1}$ and $\omega_2 = \pi -\omega^{a_2}$ ($\omega^{a_i}$ as in
            Definition~\ref{parametro_asintotico_omega}) verify 
             \begin{equation}\label{cosinn}
              \sin(\omega_1(u) + \omega_2(v))>0 \hspace{1cm} \forall (u,v)\in \R^2.
             \end{equation}
            \item[iv)] The map $(x(u,v), y (u,v))$ given by \eqref{camcor} is a
            global diffeomorphism.
        \end{itemize}
    Then, $f:\R^2\longrightarrow \mathbb{H}^3_1$ defined by
        \begin{equation}\label{est4}
        f(u,v) = a_1(u)\overline{a_2(v)}
        \end{equation}
    describes an isometric immersion of $\mathbb{L}^2$ into $\mathbb{H}^3_1$,
    and $(u,v)$ are the global characteristic parameters given in Proposition~\ref{cara}.
    \end{teo}
\begin{proof}
By definition of $\omega_1(u), \omega_2 (v)$, the condition \eqref{cosinn} and the ambiguity
in Definition \ref{parametro_asintotico_omega}, it is clear that we can suppose that $\omega_1(u) +\omega_2(v)\in (0,\pi)$. Moreover, we have
\begin{eqnarray}
    \overline{a_1(u)}a_1'(u) &=& \cos(\omega_1(u))i + \sin(\omega_1(u)) k,\label{a_1'(segunda)}\\
    \overline{a_2(v)}a_2'(v) &=& -\cos(\omega_2(v))i + \sin(\omega_2(v))k,
    \end{eqnarray}
and, conjugating the last expression,
    \begin{equation}\label{a_2'(segunda)}
    \overline{a_2'(v)}a_2(v) = \cos(\omega_2(v))i -
    \sin(\omega_2(v))k.
    \end{equation}
Hence, from \eqref{est4}, \eqref{a_1'(segunda)} and \eqref{a_2'(segunda)} we obtain

$$\esiz f_u(u,v),f_u(u,v)\esde = \esiz a_1'(u),a_1'(u)\esde =  -\cos (2\omega_1 (u)),$$
$$\esiz f_v(u,v),f_v(u,v)\esde = \esiz \overline{a_2'(v)},\overline{a_2'(v)}\esde =-\cos (2\omega_2 (v)),$$
and
    \begin{equation}
    \begin{aligned}
    \big\langle f_u&(u,v), f_v(u,v) \big\rangle =\\
        &= \big\langle \cos(\omega_1(u))i + \sin(\omega_1(u)) k \hspace*{2pt},
    \hspace*{2pt} \cos(\omega_2(v))i - \sin(\omega_2(v))k
    \big\rangle\\
    &= -\cos(\omega_1(u))\cos(\omega_2(v)) -
    \sin(\omega_1(u))\sin(\omega_2(v))\\
    &= - \cos (\omega_1(u) - \omega_2(v))
    \end{aligned}\label{F}
    \end{equation}

Now, to find the expression of the second fundamental form in
coordinates $(u,v)$, we take into account that, by \eqref{a_1'(segunda)} and \eqref{a_2'(segunda)},
    \[
    \begin{aligned}
    &f_u(u,v)\times f_v(u,v) =\\
    &=\Big(a_1(u)\big(\cos(\omega_1(u))i + \sin(\omega_1(u))
    k\big)\hspace*{1pt}\overline{a_2(v)}\Big) \times \Big(a_1(u)\big(\cos(\omega_2(v))i -
    \sin(\omega_2(v))k\big)\hspace*{1pt}\overline{a_2(v)}\Big)\\
     &= \sin(\omega_1(u) + \omega_2(v))\hspace*{3pt} a_1(u)\hspace*{1pt}
    j\hspace*{2pt} \overline{a_2(v)}
    \end{aligned}
    \]
and so, $$N(u,v) = \dfrac{f_u(u,v)\times
f_v(u,v)}{\|f_u(u,v)\times f_v(u,v)\|} = a_1(u)\hspace*{1pt}
    j\hspace*{2pt} \overline{a_2(v)}.$$
After that, we obviously get
    \begin{equation}\label{eg}
    \langle f_u(u,v),N_u(u,v)\rangle = 0
    = \langle f_v(u,v),N_v(u,v)\rangle,
    \end{equation}
and we deduce also
    \begin{eqnarray}
    \begin{aligned}
    \langle N_u(u,v), N_u(u,v)\rangle
    &= \cos(2\hspace*{2pt}\omega_1(u)),
    \end{aligned}\label{productofsfs}\\
    \begin{aligned}
    \langle N_v(u,v), N_v(u,v)\rangle &=  \cos(2\hspace*{2pt}\omega_2(v)).
    \end{aligned}\label{productoftft}
    \end{eqnarray}
Besides, it follows immediately from \eqref{a_1'(segunda)},
\eqref{a_2'(segunda)} that
\begin{equation*}
        \langle f_v(u,v), N_u(u,v)\rangle = - \sin(\omega_1(u) +
        \omega_2(v)).
        \end{equation*}
This completes the proof, using Proposition \ref{cara}.

\end{proof}

\section{The classification results}

In this section we will improve the representation formula for
flat surfaces in $\H_1^3$ in Theorem \ref{representation2}, by presenting a
geometric method to describe the curves in $\H_1^3$ verifying the
condition $\esiz a',a \, j\esde =0$. As a consequence, we will
obtain the main classification results of this paper.

Let us start by considering the unit tangent bundle to $\H^2$,
$$TU(\H^2)=\{(x,y): x\in \H^2, y\in \S_1^2, \esiz x,y\esde=0\},$$
where we are viewing here $\H^2 = \H_1^3\cap \{x_0=0\}$ and
$\S_1^2 =\S_2^3\cap \{x_0=0\}.$ With this, we can consider the map

$$\pi:\H_1^3\flecha TU (\H^2)$$ \begin{equation}\label{2eces}
x\mapsto (x \, i \, \bar{x}, x \, k\, \bar{x}) = (h_i (x),h_k (x)).
\end{equation} This map is a double covering map with $\pi (-x)=\pi (x)$ for every $x\in \H_1^3$.

From now on, let us use the notation $$h(x):= h_i (x)= x\, i \, \bar{x} :\H_1^3 \flecha \H^2.$$

\begin{defi}
A \emph{Legendrian curve} in $TU(\H^2)$ is an immersion
$\alfa=(\gamma,\nu):I\subset \R \flecha TU(\H^2)$ such that
$$\esiz \gamma', \nu\esde =0.$$
\end{defi}
Associated to such a Legendrian curve we may define the
metric $$\esiz d \alfa,d \alfa\esde_{\cS} :=\esiz
d\gamma, d\gamma\esde + \esiz d\nu, d\nu \esde.$$ As $\esiz
\gamma',\gamma'\esde \geq 0$ and $\esiz \nu',\nu'\esde \geq 0$,
and $\alfa$ is an immersion, we have that $\esiz
\alfa',\alfa'\esde_{\cS} >0$ everywhere. In particular, we may
parametrize $\alfa$ by its arclength parameter with respect to
$\esiz, \esde_{\cS}.$

In what follows, let $p_{\H^2} : TU(\H^2)\flecha \H^2$ denote the
canonical projection of $TU(\H^2)$ onto $\H^2$.

\begin{defi}
A \emph{wave front} (or simply a \emph{front}) in $\H^2$  is a
smooth map $\gamma:I\subset \R\flecha \H^2$ that lifts to a
Legendrian curve, i.e. there exists a Legendrian curve
$\alfa:I\subset \R\flecha TU(\H^2)$ such that $p_{\H^2} (\alfa) =
\gamma$. In this conditions, we call the map $\nu :I\subset
\R\flecha \S_1^2$ such that $\alfa = (\gamma,\nu)$ the \emph{unit
normal} of the front.

A \emph{closed front} in $\H^2$ is defined similarly as the
projection of a closed Legendrian curve $\alfa:\S^1 \flecha TU
(\H^2)$.
\end{defi}
It is clear that any regular curve in $\H^2$ is a front, but the
converse it not true in general. For instance, the parallel curves
of a regular curve in $\H^2$ are fronts which have singularities,
in general. Besides, there are periodic curves in $\H^2$ with singularities that are not
closed fronts with the above definition, since they do not have a globally well defined unit normal (e.g. a closed curve with exactly one cusp). For more details about fronts, see \cite{SUY,MuUm,KUY,KRSUY}.

The next lemma provides an important simplification to the
equation $\esiz a',a \, j\esde =0$.

\begin{lem}\label{lemgor}
Let $a(u):\R\flecha \H_1^3$ be a regular curve. The following
statements are equivalent:

 \begin{enumerate}
 \item[(1)]
$\esiz a'(u),a(u)\, j\esde =0$.
 \item[(2)]
$\pi (a(u)):\R\flecha TU(\H^2)$ is a Legendrian curve ($\pi$ as in
\eqref{2eces}).
 \item[(3)]
$\gamma (u):= h(a(u))$ is a front in $\H^2$ with unit normal $\nu
(u)= a(u) k \overline{a(u)}$.
 \end{enumerate}

\end{lem}
\begin{proof}
It is immediate from the definition of front in $\H^2$ and
\eqref{2eces} that \emph{(2)} and \emph{(3)} are equivalent. To
prove that $\emph{(2)}\Rightarrow \emph{(1)}$, assume that $$\pi
(a(u))=(a(u) i \overline{a(u)}, a(u) k \overline{a(u)}) $$ is
Legendrian. Then $a'(u)\neq 0$ and $$\esiz a'(u) i \overline{a(u)}
+ a(u) i \overline{a'(u)}, a(u) k \overline{a(u)}\esde =0.$$ Using
that $k \, i = j$ and the left-right invariance, this equation
gives $$\esiz a'(u), a(u) \, j\esde = \esiz i \overline{a'(u)}, k
\overline{a(u)}\esde = \esiz a'(u) \, i, a(u)\, k\esde = -\esiz
a'(u), a(u)\, j\esde,$$ i.e. \emph{(1)} holds.

To prove that $\emph{(1)}\Rightarrow \emph{(2)}$, we define
$\alfa(u):= \pi (a(u)) =(\gamma(u), \nu(u))$, i.e. $\gamma(u)=
a(u) i \overline{a(u)}$ and $\nu (u)= a(u) k \overline{a(u)}$. Let
us assume that $u$ is the \emph{asymptotic parameter} of $a(u)$ as
in Definition \ref{parametro_asintotico_omega}. Then, by
\eqref{definicion_omega} we have (omitting the parameter $u$ for
clarity)

\begin{equation}\label{2es1}
\bar{a} a' = \cos(\omega^a) \, i + \sin (\omega^a ) \, k.
\end{equation}
Hence, $$\esiz \gamma',\nu\esde = \esiz \bar{a} a' \, i + i
(\overline{\bar{a} a'}), k\esde =0.$$ So, to prove \emph{(2)} we
only have left to check that $\pi (a(u))$ is an immersion, i.e.
$\esiz \gamma',\gamma'\esde + \esiz \nu',\nu'\esde >0$ everywhere.
We compute
\begin{equation}\label{eee}
\def\arraystretch{1.5}\begin{array}{lll}
\esiz \gamma',\gamma'\esde & =& 2\esiz a',a'\esde + 2 \esiz  \bar{a}a' i, i(\overline{\bar{a} a'})\esde = \hspace{0.5cm} \text{ (by \eqref{2es1})} \\
& = & -2 \cos (2 \omega^a) + 2 \esiz (-\cos (\omega^a) \,1 + \sin
(\omega^a) \, j , \cos (\omega^a) 1 + \sin (\omega^a) j \esde \\ &
= & 4 \sin^2 (\omega^a).
\end{array}
\end{equation}
A similar computation using the general relations $\esiz xk,xk\esde= -\esiz x,x\esde = \esiz kx,kx\esde$ gives $\esiz \nu',\nu'\esde = 4\cos^2 (\omega^a)$, and consequently

\begin{equation}\label{2es2}
\esiz \alfa'(u),\alfa'(u)\esde_S = \esiz
\gamma'(u),\gamma'(u)\esde + \esiz \nu'(u),\nu'(u)\esde =4.
\end{equation}
This yields \emph{(2)} and completes the proof.
\end{proof}
Using this result, we may give the following definition.
\begin{defi}
Let $\gamma:I\subset \R\flecha \H^2$ be a front in $\H^2$ with
Legendrian lift $\alfa:I\subset \R\flecha TU(\H^2)$. An
\emph{asymptotic lift} of $\gamma$ is a regular curve $a:I\subset
\R\flecha \H_1^3$ such that $\pi \circ a = \alfa$, where
$\pi:\H_1^3\flecha TU(\H^2)$ is the double cover \eqref{2eces}.
\end{defi}

It is obvious that any front has an asymptotic lift, which is
unique up to sign once we fix the Legendrian lift $\alfa$ (since
$\pi$ is a double covering with $\pi (x)= \pi (-x)$). Also, by
Lemma \ref{lemgor}, the asymptotic lift of $\gamma(u)$ verifies
$h(a(u))= \gamma(u)$ and $\esiz a'(u),a(u)\, j\esde =0$.

Let us also observe that if we substitute the unit normal $\nu$ of the front $\gamma$ by $-\nu$, then the asymptotic lift $a(u)$ switches to $a(u)\, i$.

\begin{remark}\label{asintotico_arco_sasaki}
\emph{By \eqref{2es2}, we see that the asymptotic parameter
$u$ of the asymptotic lift $a(u)$ according to Definition
\ref{parametro_asintotico_omega} is one half of the arc-length
parameter w.r.t. the metric $\esiz,\esde_S$ of the
Legendrian lift $\alfa(u)$ of $\gamma(u)$.}
\end{remark}

Let $\gamma(u):I\subset \R\flecha \H^2$ be a front with unit
normal $\nu:I\subset \R\flecha \S_1^2$. If $\gamma'(u_0)\neq 0$,
its geodesic curvature at that point is $$k_g (u_0)= \frac{\esiz
\gamma''(u_0),\nu (u_0)\esde}{||\gamma'(u_0)||^2}.$$ Now, if
$\gamma'(u_0)=0$, then $\nu'(u_0)\neq 0$ around $u_0$, and we have
$\gamma'(u)=\landa (u) \nu'(u)$ for some smooth function
$\landa(u)$ defined in a neighborhood of $u_0$. Clearly, $\landa
(u_0)=0$ and $\landa = -1/k_g$ at regular points of $\gamma$. This
justifies the following definition:

\begin{defi}
Let $\gamma(u):I\subset \R\flecha \H^2$ be a front with unit
normal $\nu:I\subset \R\flecha \S_1^2$. The \emph{geodesic
curvature} of $\gamma$ is the smooth map $k_g:I\subset \R\flecha
\R\cup \{\8\}\equiv \R P^1$ given by
$$\left\{\def\arraystretch{1.7}
\begin{array}{cll} k_g (u)= \displaystyle\frac{\esiz \gamma''(u),\nu (u)\esde}{||\gamma'(u)||^2} & \text{ if } & \gamma'(u) \neq 0 \\ \8 & \text{ if } & \gamma'(u) =0.
\end{array}\right.$$
\end{defi}

The geodesic curvature of a front in $\H^2$ and the angle function
of its asymptotic lift in $\H_1^3$ are related by the following
simple formula:

\begin{lem}\label{curvatura_es_cot(omega)}
Let $a(u):I\subset \R\flecha \H_1^3$ be a regular curve in
$\H_1^3$ with $\esiz a'(u),a(u)\, j\esde =0$, where $u$ is its
asymptotic parameter. Then, the geodesic curvature $k_g(u)$ of the
front $\gamma (u)=h(a(u)):I\subset \R\flecha \H^2$ is given by
 \begin{equation}\label{eces4}
 k_g (u)= {\rm cot} (\omega^a(u)),
 \end{equation}
where $\omega^a (u)$ is the angle function of $a(u)$ (see Definition \ref{parametro_asintotico_omega}).
\end{lem}
\begin{proof}
We know by \eqref{eee} that $\gamma'(u_0)=0$ if and only if $\sin
\omega^a(u_0)= 0$. Thus, \eqref{eces4} holds trivially at the
singular points of $\gamma(u)$.

For the rest of the points, we use \eqref{eee}, the equivalence
$\emph{(1)}\Leftrightarrow \emph{(3)}$ in Lemma \ref{lemgor}, and
the left-right invariance to compute $$k_g =-\frac{\esiz
\gamma',\nu'\esde}{||\gamma'||^2} = \frac{-1}{4 \sin^2 (\omega^a)}
\, \esiz \bar{a} a' \, i + i \, (\overline{\bar{a} a'}),\bar{a} a'
\, k + k \, (\overline{\bar{a} a'})\esde.$$ Using now \eqref{2es1}
and the relations $i\, k = -k \, i =-j$ and $k^2 = -i^2 =1$, we
have $$k_g = \frac{-1}{4 \sin^2 (\omega^a)} \, \esiz 2 \sin
(\omega^a ) \, j, -2 \cos (\omega^a) \, j\esde = {\rm cot}
(\omega^a),$$ as desired.
\end{proof}

    \begin{remark}\label{definicion_cot-1} \emph{Let us note that the function $
    \cot:\R \rightarrow\R P^1$ is a continuous, surjective, $\pi$-periodic covering map. This allows us to choose, on every subset
    $A\subsetneq \R P^1$ a continuous determination of ${\rm cot}^{-1}$ such that:}
        \begin{eqnarray*}
        \cot^{-1}(A) &\subset& (0, \pi) \hspace*{3mm}\text{ if } \hspace*{1mm}\infty \notin A,\\
        \cot^{-1}(A) &\subset& (\pi - c, 2\pi - c) \text{ for some } c\in (0, \pi) \hspace*{2mm}\text{ if }\hspace*{1mm}\infty \in A.
        \end{eqnarray*}
\emph{From now on, by $\cot^{-1}$ we shall mean this specific continuous determination.}
    \end{remark}

\begin{lem}\label{posun}
Let $\gamma:I\subset \R\flecha \H^2$ be a front with unit normal $\nu:I\subset \R\flecha \S_1^2$, whose geodesic curvature function $k_g:I\subset \R\flecha \R P^1$ is not surjective onto $\R P^1$ (this holds, for instance, if $\gamma$ is regular). Then, changing $\nu$ to $-\nu$ if necessary, $\gamma$ admits an asymptotic lift $a:I\subset \R\flecha \H_1^3$ such that:

\begin{enumerate}
\item
If $\gamma$ is regular, then $\sin (\omega^a)>0$. This happens with respect to the unit normal $\nu$ such that $\{\gamma',\nu\}$ is always a positively oriented basis of $T_{\gamma} \H^2$.
 \item
If $\gamma$ is not regular, then $\overline{a(u_0)}\, a'(u_0) = -i$ (i.e. $\cos (\omega^a (u_0)) = -1$) for all points with $\gamma'(u_0)=0$.
\end{enumerate}
If this holds for the unit normal $\nu$, then $\nu$ will be called the \emph{positive unit normal} of the front $\gamma$.
\end{lem}
\begin{proof}
If $\gamma$ is regular, then by \eqref{eee} we have $\sin
(\omega^a)\neq 0$ everywhere. Now, observe that if we change $\nu$
by $-\nu$, the asymptotic lift $a(u)$ changes to $a(u) \, i$, and
so, by \eqref{2es1}, $\sin (\omega^a)$ changes to $-\sin
(\omega^a)$.

Now, let $\nu$ denote the unit normal of $\gamma$ for which $\sin
(\omega^a)>0$ (which exists by the above explanation). Then, by
Lemma \ref{lemgor}, $\{\gamma',\nu\}$ will be a positively oriented
basis of $T_{\gamma} \H^2$ if and only if $\esiz \gamma',a\, j \,
\bar{a}\esde >0$ at every point (observe that $\{ai \bar{a}, aj \bar{a}, ak \bar{a}\}$ is always positively oriented). Now, from \eqref{2es1} we get
$$\def\arraystretch{1.4}\begin{array}{lll}
\esiz \gamma',a \,j \, \bar{a}\esde &=& \esiz a' \, i\, \bar{a} +
a\,  i\, \overline{a'},a\, j\, \bar{a}\esde = \esiz \bar{a}
a' \,i,j\esde + \esiz i \, (\overline{\bar{a} a'}),j\esde \\ &=&
-2\esiz \bar{a} a',ji\esde = 2\sin (\omega^a)>0,
\end{array}$$
what proves the claim.

Now, assume that $\gamma'(u_0)=0$ for some $u_0$. Then $\sin
(\omega^a (u_0))=0$ and, changing $\nu$ by $-\nu$ (and thus $a(u)$
by $a(u) \, i$) if necessary, we may assume that $\cos
(\omega^a(u_0))=-1$. Now, by Remark \ref{definicion_cot-1} and the hypothesis
that $k_g$ is not surjective onto $\R P^1$, the claim that $\cos
(\omega^a(u))=-1$ actually holds at every singular point of the
front $\gamma$. This concludes the proof.
\end{proof}
   \begin{defi}\label{admis}
An \emph{admissible front pair} in $\H^2$ is a pair of fronts $\gamma_1,\gamma_2:\R\flecha \H^2$ with $\gamma_1 (0)=\gamma_2(0)=i$ and $\nu_1(0)=\nu_2(0)=k$, such that
 \begin{enumerate}
 \item[i)]
$\gamma_1$ is actually a regular curve in $\H^2$.
 \item[ii)]
If $k_1,k_2:\R\flecha \R P^1$ denote the geodesic curvatures of $\gamma_1$ and $\gamma_2$, respectively, with respect to their positive unit normals, then $$k_1(u)\neq k_2(v) \hspace{1cm} \forall (u,v)\in \R^2,$$ and actually $k_1(u)>k_2(v)$ holds if $\gamma_2$ is also a regular curve.
 \end{enumerate}
 \end{defi}
We observe that if $\gamma_1,\gamma_2$ verify $k_1(\R)\cap k_2(\R) = \emptyset$, then by switching the roles of $\gamma_1$ and $\gamma_2$ if necessary, $\{\gamma_1,\gamma_2\}$ is an admissible front pair in $\H^2$.

These elements will let us describe in a very precise way the
moduli space of isometric immersions of $\LO^2$ into $\H_1^3$ in
terms of suitable pairs of curves with front-like singularities in
$\H^2$. Indeed, we have

    \begin{teo}[Classification of complete examples]\label{clas1}
        Let $\gamma_1 (u),\gamma_2(v):\R\flecha \H^2$ be an admissible front pair in $\H^2$, where $u/2$ (resp. $v/2$) is the arc-length parameter of $\gamma_1$ (resp. $\gamma_2$) with respect to the metric $\esiz,\esde_{\mathcal{S}}$. 

Let $k_1(u), k_2(v):\R\rightarrow \R P^1$ and $a_1(u), a_2(v):\R\rightarrow \H^3_1$ denote, respectively, the geodesic curvatures and asymptotic lifts of $\gamma_1$ and $\gamma_2$ with respect to their positive unit normals. Assume that:
 \begin{itemize}
 \item For $\omega_1 (u):= \cot^{-1}(k_1(u))$ and $\omega_2 (v):= \pi - \cot^{-1}(k_2(v))$, the map $(x(u,v), y (u,v))$ defined in \eqref{camcor} is a global diffeomorphism.
            \end{itemize}
  Then, $f:\R^2\longrightarrow \mathbb{H}^3_1$ given by $f(u,v) = a_1(u)\overline{a_2(v)}$
        is an isometric immersion of $\mathbb{L}^2$ into $\mathbb{H}^3_1$,
        and $(u,v)$ are the global characteristic parameters given in Proposition~\ref{cara}.

        Conversely, every isometric immersion of $\mathbb{L}^2$ into $\mathbb{H}^3_1$ can be recovered by this process from an admissible front pair in $\H^2$.
    \end{teo}

    \begin{proof}
        For the direct part, we just have to show that $a_1(u)$ and $a_2(v)$ verify the
        hypotheses of Theorem~\ref{representation2}. Since they are the asymptotic lifts
        of the curves $\gamma_i$, Lemma~\ref{lemgor} tells us that
        $\langle a_1', a_1 j \rangle = \langle a_2', a_2 j \rangle= 0$ and, by
        Remark~\ref{asintotico_arco_sasaki},
        we know that $u$ and $v$ are the asymptotic parameters of $a_1$ and $a_2$. Also by Lemma \ref{lemgor} and the sign ambiguity of the asymptotic lift, we may assume that $a_1(0)=a_2(0)=1$.

        Now, observe that condition $ii)$ in Definition \ref{admis} implies, in particular, that both $k_1(\R), k_2(\R)\subsetneq \R P^1$. So, by Remark \ref{definicion_cot-1} the functions $\cot^{-1}(k_1(u))$ and $\cot^{-1}(k_2(v))$ make sense. 

Let $\omega^{a_1}$ (resp. $\omega^{a_2}$) denote the angle function associated to $a_1$ (resp. $a_2$). As $\gamma_1$ is regular, by Lemma \ref{posun} and the $2\pi k$-ambiguity in defining $\omega^{a_i}$, we may assume that $\omega^{a_1} (\R)\subset (0,\pi)$. Thus, by Lemma \ref{curvatura_es_cot(omega)} and the above comments we have $$\omega^{a_1} (u)=\cot^{-1} (k_1(u)) \in (0,\pi),$$ and similarly, $$\omega^{a_2} (v)= \cot^{-1} (k_2(v)),$$ where $\omega^{a_2} (\R) \subset (0,\pi)$ if $\gamma_2$ is regular, and $\omega^{a_2} (\R) \subset (\pi-c,2\pi-c)$ for some $c>0$ if $\gamma_2$ has some singular point.

Define now $\omega_1(u) = \cot^{-1}(k_1(u))$ and $\omega_2(v) = \pi - \cot^{-1}(k_2(v))$. If we prove that $\omega_1(u) +\omega_2(v)\in (0,\pi)$ for all $(u,v)\in \R^2$, all conditions of Theorem \ref{representation1} will be fulfilled, as we wished.

In case that $\gamma_2$ is regular, we clearly have $\omega_1(u) +\omega_2(v)>0$, and as $k_1(u)>k_2(v)$ for every $(u,v)$, we conclude that $$\cot^{-1} (k_1(u)) -\cot^{-1} (k_2(v)) <0,$$ i.e. $\omega_1 (u)+ \omega_2(v)<\pi$, as desired.

In case that $\gamma_2$ has some singular point, it is clear that $\omega_1 (u_0)+\omega_2 (v_0) \in (0,\pi)$ for some adequate $(u_0,v_0)\in \R^2$. Once we know that, it is also clear that $\omega_1(u)+ \omega_2(v)\neq \{0,\pi\}$ at every point, since otherwise the condition $k_1(u)\neq k_2(v)$ would not hold everywhere. So, again, $\omega_1 (u)+\omega_2(v)\in (0,\pi)$ for all $(u,v)\in \R^2$. This finishes the first part of the proof.

        For proving the converse part of the theorem we recall that, from Theorem~\ref{representation1},
        we already know that every isometric immersion of $\mathbb{L}^2$ into $\mathbb{H}^3_1$
        can be put in the form $f(u,v) = a_1(u)a_2(v)$. Thus, taking
        $\gamma_1(u) = p_{\H^2}\circ\pi(a_1(u))$, $\gamma_2(v) = p_{\H^2}\circ\pi\big(\hspace*{1pt}\overline{a_2(v)}\hspace*{1pt}\big)$, we can recover the immersion $f$ by applying the direct part to the curves $\gamma_1$ and $\gamma_2$.
    \end{proof}

Let us now consider a Lorentzian flat surface $\Sigma$ which is
compact and orientable. Then $\Sigma$ is a torus and its universal
covering $\tilde{\Sigma}$ is a plane. The next classification
result establishes which isometric immersions of $\LO^2$ into
$\H_1^3$ in Theorem \ref{clas1} are the universal covering of some
Lorentzian flat torus in $\H_1^3$. This provides then a
description of the moduli space of Lorentzian flat tori in
$\H_1^3$.

\begin{teo}[Classification of flat tori]\label{clas2}
Let $\gamma_1,\gamma_2:\S^1\flecha \H^2$ be two closed fronts in $\H^2$ with $\gamma_1 (p_0)=\gamma_2(p_0)=i$ and $\nu_1 (p_0)=\nu_2(p_0)=k$ for some $p_0\in \S^1$ (here $\nu_i$ is the positive unit normal of $\gamma_i$). Assume that
  \begin{equation}\label{regc}
k_1 (\S_1) \cap k_2 (\S_1) = \emptyset,
  \end{equation}
where here $k_i$ is the geodesic curvature of $\gamma_i$ in
$\H^2$. Then, permuting $\gamma_1$ and $\gamma_2$ if necessary, the Lorentzian flat surface in $\H_1^3$ that they
generate via Theorem \ref{clas1} has compact image, and describes
therefore a Lorentzian flat torus isometrically immersed in
$\H_1^3$.

Conversely, every Lorentzian flat torus of $\H_1^3$ can be
constructed following the process described in Theorem
\ref{clas1}, starting with a pair of closed fronts
$\gamma_1,\gamma_2$ in $\H^2$ satisfying the regularity condition
\eqref{regc}.

\end{teo}

\begin{proof}
The first part is immediate, taking into account that if
$\gamma_i$ is a closed front with unit normal $\nu_i$, then $\alfa_i
:=(\gamma_i, \nu_i)$ is regular and closed in
$TU(\H^2)$, and as $\pi$ in \eqref{2eces} is a double covering, it follows
that $a_i:=\pi^{-1} ( \alfa_i)$ will be a closed curve
in $\H_1^3$. With this, $f=a_1 \overline{a_2}$ is the product of
two closed curves in $\H_1^3$, and thereby it is compact with the
topology of a torus.

Conversely, let $\Sigma$ denote a flat Lorentzian torus in
$\H_1^3$, let $\tilde{\Sigma} \equiv \LO^2$ denote its universal
covering, and $p:\tilde{\Sigma}\flecha \Sigma$ the canonical
covering map. So, we shall regard $\tilde{\Sigma}$ in the obvious
way as a complete Lorentzian flat surface isometrically immersed
in $\H_1^3$, with second fundamental form $\tilde{II}$ given by
$p^* (II)=\tilde{II}$, where $II$ stands for the second
fundamental form of the torus $\Sigma$. In these conditions, by
Theorem \ref{clas1}, we can parametrize $\tilde{\Sigma}$ as an
immersion $f(u,v):\R^2\flecha \H_1^3$ such that $$f(u,v)=a(u) \,
\overline{b(v)}, \hspace{1cm} N(u,v)= a(u) \, j \,
\overline{b(v)}.$$ Here we have assumed that, up to a rigid
motion, $f(0,0)= 1$ and $N(0,0)=j$.

Let us consider next the map $$N \, \bar{f} :\tilde{\Sigma}\flecha
\S_1^2.$$ It is obvious that $N \, \bar{f}$ is well defined in
$\Sigma$, and thus $N \, \bar{f} (\Sigma)= N \, \bar{f}
(\tilde{\Sigma})$ is compact in $\H^2$. Moreover, in terms of the
parameters $(u,v)$ we have $$(N \, \bar{f}) (u,v)= a(u) \, j
\, \overline{a(u)},$$ and hence $N \, \bar{f} (\tilde{\Sigma})$ is
a closed curve in $\S_1^2$. We prove next that it is also regular.

Let us denote $\beta_1 (u):= a(u) \,j \,
\overline{a(u)}:\R\flecha \S_1^2$. Then, using the basic properties
of the pseudo-quaternionic model for $\H_1^3$, we have

$$\def\arraystretch{1.3}\begin{array}{lll}
\overline{\beta_1}\, \beta_1' & = & - a \,j\, \bar{a} (a'
\,j\, \bar{a} + a\, j\, \overline{a'})
 =  -a \,j\, \bar{a}\, a'\, j\, \bar{a} - a\, \overline{a'} \\
 & = & -a\, j\, \bar{a}\, a\, j\, \overline{a'} - a\, \overline{a'} \hspace{1cm} (\text{ since } \esiz a',a \, j \esde =0 \Rightarrow {\rm Re } (a' \, j \, \bar{a})=0) \\
  & = & -2 a \,\overline{a'} \neq 0.
\end{array} $$ Therefore, $\beta_1(u)$ is a regular curve, which is also closed.

In the same way, we can define $$-\bar{N} \, f
:\tilde{\Sigma}\flecha \S_1^2,$$ and the process above shows that
the curve $\beta_2(v):= b(v) \, j\, \overline{b(v)}
:\R\flecha \S_1^2$ is a closed regular curve.

It is important to remark that, by its own construction, the
curves $\beta_i$ may be seen as defined on the flat torus
$\Sigma$. Consider next the map
$$G=(\beta_1,\beta_2):\Sigma\flecha \S_1^2\times \S_1^2.$$ It is
obvious that $G$ is a local diffeomorphism, and $G(\Sigma)\equiv
\beta_1 \times \beta_2 \subset \S_1^2\times \S_1^2$ is a (flat) torus.
Thus, by compactness, $G$ is a finite folded covering map. In this
way, the lift to $\Sigma$ of each curve of the form $\Gamma:=
\beta_1\times \{p\}$ or $\Gamma:=\{p\}\times \beta_2$ of the torus
$\beta_1\times \beta_2$ is a closed curve in $\Sigma$.

In addition, it is clear from the definition of $\beta_1,\beta_2$ that a
curve $\tilde{\alfa}$ is an asymptotic curve on $\tilde{\Sigma}$
(if and only if $\alfa = p \circ \tilde{\alfa}$ is an asymptotic
curve on $\Sigma$) if and only if $\tilde{G_i}\circ \tilde{\alfa}$
is constant for some $i=1,2$, where by definition $\tilde{G_i}=
G_i \circ p$. Thus, $\alfa$ is an asymptotic curve on $\Sigma$ if
and only if $G_i\circ \alfa$ is constant for some $i=1,2$, i.e. if
and only if $\alfa$ is the lift via the finite fold covering $G$
of a curve of the form $\Gamma:= \beta_1\times \{p\}$ or
$\Gamma:=\{p\}\times \beta_2$ on $\beta_1\times \beta_2$.

To sum up, we have proved the fundamental fact that \emph{the
asymptotic curves of a Lorentzian flat torus in $\H_1^3$ are
closed}. In particular, the Hopf projection into $\H^2$ of such an
asymptotic curve is a closed front. This fact together
with the converse part of Theorem \ref{clas1} proves that every
Lorentzian flat torus in $\H_1^3$ can be reconstructed by means of
two closed fronts in $\H^2$ verifying the regularity
condition \eqref{regc}. This completes the proof.
\end{proof}

\begin{remark}
\emph{Theorem \ref{clas2} constitutes the extension to the Lorentzian setting of the classification of Riemannian flat tori in the $3$-sphere $\S^3$ by Kitagawa \cite{Kit1}. Let us remark that Theorem \ref{clas2} follows from our main result (Theorem \ref{clas1}) and a reformulation of the proof of Kitagawa's theorem given by Dadok and Sha in \cite{DaSh}.}
\end{remark}

\section*{Hopf cylinders}

The most simple examples of isometric immersions from $\LO^2$ into $\H_1^3$ are provided by Hopf cylinders. Next, we will analyze these Hopf cylinders from the viewpoint developed in this paper.

Let us denote by $\Lambda^2$ the positive light cone $\Lambda^2 =\{x \in \LO^3 : \esiz x,x\esde=0, x_0 >0\}$.

    \begin{defi}Let $\sigma$ be a regular (spacelike or timelike) curve in $\S^2_1$ (resp. $\H^2$, $\Lambda^2$) and $\rho\in\R^4_2$ be pure imaginary and nonzero with $\langle \rho, \rho \rangle = 1$ (resp. $\langle \rho, \rho \rangle = -1$, $\langle \rho, \rho \rangle = 0$). Then the flat surface in $\H^3_1$ given by $M_\rho(\sigma)=h_\rho^{-1} (\sigma)$ is called a \emph{Hopf cylinder}.
    \end{defi}
The Hopf cylinders $M_\rho(\sigma)$ with $\langle \rho, \rho \rangle = -1$ or $\langle \rho, \rho \rangle = 0$ are always timelike, whereas those with $\langle \rho, \rho \rangle = 1$ can be both spacelike or timelike, depending on the causal character of the curve $\sigma$. Moreover, if $\sigma$ is a closed curve in $\H^2$, the Hopf cylinder $M_\rho(\sigma)$ is actually compact, and is called a \emph{Lorentzian Hopf torus}.

Since complete Lorentzian Hopf cylinders are particular cases of isometric immersions of $\LO^2$ into $\H^3_1$, Theorem~\ref{clas1} tells us that they can be obtained from two curves  $\gamma_1, \gamma_2$ with front singularities in $\H^2$. In this situation one may ask whether there exists any condition on the curves $\gamma_i$ which characterizes Lorentzian Hopf cylinders among all isometric immersions of $\LO^2$ into $\H_1^3$.

    \begin{teo}
    Let $f:\LO^2\longrightarrow \H_1^3$ be an isometric immersion which is a Lorentzian Hopf cylinder $M_\rho(\sigma)$. We assume that $f(0,0) = 1$ and $N(0,0) = j$ (this forces $\langle \rho, j\rangle = 0$). Then, $f$ can be recovered following the process described in Theorem~\ref{clas1} from two fronts $\gamma_1$, $\gamma_2$ in $\H^2$ such that at least one of them has constant geodesic curvature $k_i$. Moreover,
        \begin{equation}\label{curvatura_caracter_rho}
        \begin{aligned}
        |k_i| > 1 \quad &\Longleftrightarrow \quad\langle \rho, \rho \rangle = -1\\
        |k_i| = 1 \quad &\Longleftrightarrow \quad\langle \rho, \rho \rangle = 0\\
        |k_i| < 1 \quad &\Longleftrightarrow \quad\langle \rho, \rho \rangle = 1
        \end{aligned}
        \end{equation}
    \end{teo}

    \begin{proof}
    Let $c(t)$ be the fiber of  $h_\rho$ passing through $1 = f(0,0)$. It is a geodesic of $\H^3_1$ and, hence, an asymptotic curve of the immersion.

    After \eqref{fibras} we know that this curve is given by $c(t) = e^{t\rho}$ and it is easy to check that
        \begin{equation*}
        \overline{c(t)}c'(t) = c'(t)\overline{c(t)} = \rho.
        \end{equation*}
    If we reparametrize this curve by its asymptotic parameter $s$ (see Definition~\ref{parametro_asintotico_omega}), then, for some constant $\omega_0\in\R$, we can write
        \begin{equation}\label{derivada_curva_c}
        \overline{c(s)}c'(s) = c'(s)\overline{c(s)} = \cos(\omega_0) i + \sin(\omega_0) k.
        \end{equation}
    From this expression it is clear that
        \begin{equation}\label{relacion_rho_omega0}
        \rho = \lambda \big(\cos(\omega_0) i + \sin(\omega_0) k\big), \quad\text{with }\lambda>0.
        \end{equation}

     On the other hand, if we consider now the global characteristic parameters $(u,v)$ of the immersion $f$ described in Proposition~\ref{cara}, we can apply Theorem~\ref{representation1} and conclude that
        \begin{equation*}
        f(u,v) = a_1(u) a_2(v)\quad \text{with } a_1 = f(u,0), \hspace{0.5cm} a_2 = f(0,v).
        \end{equation*}
     Using the terminology of Theorem~\ref{clas1}, we get that the immersion $f$ can be recovered from the fronts $\gamma_1 = h(a_1)$ and $\gamma_2 = h(\overline{a_2})$ in $\H^2$.

    The fact that $c(s)$ is an asymptotic curve of $f$ passing through $f(0,0)$ implies that it is a reparametrization of one of the curves $a_i$. In this situation, the corresponding $\gamma_i$ would have constant geodesic curvature if and only if the front $h(c(s))$ (or $h(\overline{c(s)})$) does. But, applying Lemma~\ref{curvatura_es_cot(omega)}, we deduce from \eqref{derivada_curva_c} that the geodesic curvatures of $h(c(s))$ and $h(\overline{c(s)})$ are both given by
        \begin{equation*}
        k_g = \cot(\omega_0).
        \end{equation*}
        
    Finally, if we recall \eqref{relacion_rho_omega0}, we can relate the different possibilities for $\langle\rho,\rho\rangle$ with $|\cot(\omega_0)|$. Namely, we get  
        \begin{equation*}
        \begin{array}{lcccc}
        \langle \rho, \rho \rangle = -1  & \Leftrightarrow & |\cos(\omega_0 ) | > |\sin(\omega_0 )|  & \Leftrightarrow &  |\cot(\omega_0)|>1\\       
        \langle \rho, \rho \rangle = 0  & \Leftrightarrow & |\cos(\omega_0 ) | = |\sin(\omega_0 )|  & \Leftrightarrow &  |\cot(\omega_0)|=1\\       
        \langle \rho, \rho \rangle = 1  & \Leftrightarrow & |\cos(\omega_0 ) | < |\sin(\omega_0 )|  & \Leftrightarrow &  |\cot(\omega_0)| < 1
        \end{array}
        \end{equation*}      
    Therefore, \eqref{curvatura_caracter_rho} is established.
    \end{proof}

\section{The Dajczer-Nomizu questions}

The global study of isometric immersions from $\LO^2$ into $\H_1^3$ was probably initiated by M. Dajczer and K. Nomizu \cite{Dajczer} in 1981. In Theorem 7.6 of that paper, the authors presented a method to construct timelike flat surfaces in $\mathbb{H}^3_1$ by multiplying two regular
curves $b_1(s):\R\flecha \H_1^3$ and $b_2(t):\R\flecha \H_1^3$ of $\H_1^3$ satisfying some additional hypotheses. Translating their notation to our context, these hypotheses on the curves are:
    \begin{itemize}
    \item[i)] $\langle b_1', b_1'\rangle \equiv -1$, $\langle b_2', b_2'\rangle \equiv 1$, i.e. a curve is timelike and the other one is spacelike.
    \item[ii)] $b_1(0) = 1 = b_2(0)$
    \item[iii)] $\langle b_1',b_1\xi_0\rangle \equiv 0
    \equiv \langle b_2',\xi_0 b_2\rangle$
    (we may assume $\xi_0 = j$)
    \end{itemize}
In these conditions, they conclude that the surface $f:\R^2\longrightarrow\mathbb{H}^3_1$ given by
$f(s, t) = b_1(s)b_2(t)$
is a timelike flat surface for some domain $D\subset\mathbb{R}^2$ containing the origin. ($D\subset \R^2$ is the connected component of the origin of all points $(s,t)\in \R^2$ at which $f$ is an immersion).

Moreover, the curves $b_1$ and $b_2$ are the asymptotic curves of $f$.

\vspace*{5mm} After proving this result, they proposed the following global open problems related to the above construction:
    \begin{itemize}
    \item[Q1:] \emph{Is $D = \mathbb{R}^2$ if the curves $b_1$ and $b_2$ are defined on
    all $\mathbb{R}$?}
    \item[Q2:] \emph{If $D = \mathbb{R}^2$, is the surface (geodesically) complete?}
    \item[Q3:] \emph{Can every isometric immersion of \hspace*{3pt}$\mathbb{L}^2$ into \hspace*{1pt}$\mathbb{H}^3_1$
    be obtained as a product of two appropriate curves?}
    \end{itemize}
Taking into account Theorem 7.6 of \cite{Dajczer}, it is reasonable to think that Question 3 was formulated as a problem restricted to spacelike or timelike curves, although this was not explicitly stated there. So, the following problem should also be considered:
    \begin{itemize}
    \item[Q4:]\emph{Can every isometric immersion of \hspace*{3pt}$\mathbb{L}^2$ into \hspace*{1pt}$\mathbb{H}^3_1$ be obtained as a product of two curves so that one of them is
    everywhere timelike and the other is everywhere spacelike?}
    \end{itemize}

 Note that we have already given a positive answer to question
 Q3 in Theorem~\ref{representation1}. Moreover, in Theorem~\ref{clas1} we have seen
 that those two curves $a_1$ and $a_2$ can be obtained as asymptotic lifts of two
 fronts, $\gamma_1$ and $\gamma_2$, in $\mathbb{H}^2$. We know also that, if $k_i$ represents the geodesic curvature of $\gamma_i$ and we take $\omega_1 (u):= \cot^{-1}(k_1)$ and $\omega_2 (v):= \pi - \cot^{-1}(k_2)$, then
    $\langle a_i', a_i' \rangle = - \cos(2\omega_i)$. Therefore, since
    \[
    -\cos(2\omega_i) =
    \dfrac{1 - \cot^2(\omega_i)}{1 + \cot^2 (\omega_i)},
    \]
 we conclude that
    \begin{equation}\label{caracter_causal}
    \begin{array}{rcl}
    a_i'\text{ is timelike}  &\Longleftrightarrow & |k_i|>1\\
    a_i'\text{ is null}  &\Longleftrightarrow & |k_i|= 1\\
    a_i'\text{ is spacelike}  &\Longleftrightarrow & |k_i|<1
    \end{array}
    \end{equation}

    This remark gives us an easy way to find a counterexample to Q4.
We just have to take fronts $\gamma_1$ and
$\gamma_2$ in $\mathbb{H}^2$ verifying the hypotheses of Theorem~\ref{clas1}, but both with $|k_i|>1$ or both
with $|k_i|<1$. In that case, their asymptotic lifts $a_1$ and $a_2$ would generate an isometric immersion of $\LO^2$ into $\mathbb{H}^3_1$ and, according to~\eqref{caracter_causal}, they would have the same causal character.

Theorem~\ref{clas1} is also the key to provide a positive answer to
question Q1. First, let us consider two curves $b_1(s)$ and $b_2(t)$ as in Theorem 7.6 of \cite{Dajczer}. We can reparametrize $b_1$ and $\overline{b_2}$ taking
    \begin{equation*}
    a_1(u) = b_1(s(u)), \qquad a_2(v) = \overline{b_2(t(v))}
    \end{equation*}
with $u, v$ the asymptotic parameters according to
Definition~\ref{parametro_asintotico_omega}. In this way, we obtain two
curves verifying $a_1(0) = a_2(0)=1$, $\langle a_1',a_1 j\rangle =\langle a_2',a_2 j\rangle=0$ and so that $a_1$ is everywhere timelike
and $a_2$ is everywhere spacelike. 

Now, consider the fronts $\gamma_1 = h(a_1)$, $\gamma_2 = h(a_2)$ in $\H^2$ (see Lemma~\ref{lemgor}). Then, by changing the order of $\gamma_1$ and $\gamma_2$ if necessary (which simply means conjugation in the product $a_1(u)\overline{a_2(v)}$), we see that $\{\gamma_1,\gamma_2\}$ is an admissible front pair. Therefore, the Lorentzian flat surface (not necessarily complete) obtained
via Theorem~\ref{clas1} has no singular points. That is,
the immersion
    \[
    f(s,t) = b_1(s) b_2(t) = a_1(u(s)) \overline{a_2(v(t))}
    \]
is defined over all $\mathbb{R}^2$. Thus, we have answered affirmatively question Q1.

Finally, the following theorem shows that question Q2 has, in general, a negative answer. Besides, it also shows that is still possible to give some sufficient conditions in the sense of \cite{Cec} and \cite{Sas} to ensure completeness in the Lorentzian case. This was exactly the way the problem was formulated in \cite{Dajczer}, where it is claimed: \emph{This [question Q2] seems to be a much more difficult problem than the question of completeness of flat surfaces in $\S^3$ treated in \cite{Cec} and \cite{Sas}}.

\begin{teo}
Let $b_1,b_2:\R\flecha \H_1^3$ be two regular curves with $b_1(0)=b_2(0)=1$, and such that $-\esiz b_1',b_1'\esde = \esiz b_2',b_2'\esde =1$ and $\esiz b_1',b_1 j \esde = \esiz b_2',b_2 j\esde =0$. Consider the timelike flat surface $$f(s,t)= b_1(s) \, \overline{b_2(t)} :\R^2\flecha \H_1^3,$$ which has no singular points by the above explanation. Then:
 \begin{enumerate}
 \item
The asymptotic parameters $(u,v)$ of $f$ are globally defined on $\R^2$.
 \item
The surface $f$ is geodesically complete if the angle function $\omega(u,v)=\omega_1(u) + \omega_2(v)$ associated to the asymptotic parameters $(u,v)$ verifies $0<c\leq \sin \omega(u,v)$ for some $c>0$.
 \item
There exist curves $b_1,b_2$ as before so that the resulting timelike flat surface is not geodesically complete.
 \end{enumerate}
\end{teo}
\begin{proof}
Let us first prove that the parameters $(u,v)$ are globally defined on $\R^2$. We shall only prove that $u=u(s)$ is globally defined on $\R$ (the case of $v=v(t)$ is analogous). As $\langle b_1'(s), b_1(s) j\rangle \equiv 0$ and $\langle b_1'(s), b_1'(s) \rangle \equiv -1$,
we can write
    \begin{equation*}
    \overline{b(s)} b'(s) \hspace*{3pt} = \hspace*{3pt}\pm \hspace*{3pt}\cosh\hspace*{-1pt} \big(\theta(s)\big)\hspace*{1pt} i
    \hspace*{3pt}\pm\hspace*{3pt}\sinh\hspace*{-1pt} \big(\theta(s)\big)\hspace*{1pt}
    k \quad\text{for some }\theta\in\mathcal{C}^\infty(\R).
    \end{equation*}
So, the asymptotic parameter of $b_1$ is given by
    \begin{equation*}
    u(s) = \Large{\int_0^s} \sqrt{\cosh^2(\theta(r)) + \sinh^2(\theta(r))}
    \hspace*{2pt}dr.
    \end{equation*}
Hence,
    \begin{equation*}
    |u(s)| = \left|\Large{\int_0^s} \sqrt{\cosh^2(\theta(r)) + \sinh^2(\theta(r))}
    \hspace*{2pt}dr\right| \geq \left|\Large{\int_0^s} 1
    \hspace*{2pt}dr\right| = |s|.
    \end{equation*}
As $s$ is globally defined on $\R$, so is $u$. This proves the first claim. Besides, the second claim follows directly from Proposition \ref{cara}.

Finally, to prove the third claim we need to find a timelike flat surface $$f(u,v):\R^2\flecha \H_1^3$$ with globally defined asymptotic coordinates $(u,v)$, such that $\esiz f_u,f_u\esde = -1$, $\esiz f_v,f_v\esde =1$, the curves $f(u,0)$ and $f(0,v)$ are globally defined on $\R$ when parametrized by arc-length, but such that the surface is not geodesically complete.

For that, let us consider a smooth function $\omega_1(u):\R\flecha (0,\pi /4)$ verifying:
 \begin{itemize}
 \item
$\displaystyle \int_0^{\8} \sqrt{\cos (2\omega_1 (u))} du = \8$ , $\displaystyle\int_{-\8}^0 \sqrt{\cos (2\omega_1 (u))} du = \8$.
 \item
$\displaystyle\int_0^{\8} \cos (2\omega_1 (u)) du < \8$.
 \end{itemize}
Define now $\omega_2 := \pi /2 + \omega_1 :\R\flecha (\pi/2, 3\pi /4)$. By Proposition \ref{cara}, $\omega_1$ and $\omega_2$ define a timelike flat surface $f(u,v):\R^2\flecha \H_1^3$ with globally defined asymptotic parameters. Besides, by \eqref{tel}, the curves $f(u,0)$ and $f(0,v)$ are globally defined on $\R$ when parametrized by arc-length. Now, to prove that the surface is not geodesically complete, we need to ensure that the map $(x(u,v),y(u,v))$ in \eqref{camcor} is not a global diffeomorphism of $\R^2$. But by Remark \ref{remcom}, we just need to prove that the Riemannian flat metric $\tilde{I} := dx^2 + dy^2$ is non-complete. 

Now, consider the divergent line $\gamma(t)=(t,t):[0,\8)\flecha \R^2$ in the $u,v$-plane. Then, by \eqref{camcor}, we have $$\tilde{I}(\gamma'(u),\gamma'(u))= 2 (1- \sin(2 \omega_1 (u))).$$ So, noting that $$ \sqrt{1- \sin (2\omega_1 )} = \frac{\cos (2\omega_1)}{\sqrt{1+\sin (2\omega_1)}},$$ we have by the condition imposed to $\omega_1$ from the start that $$\int_0^{\8} \sqrt{\tilde{I} (\gamma',\gamma')} du < \8,$$ i.e. $\gamma$
is a divergent curve of finite length. Thus, the map \eqref{camcor} is not a global diffeomorphism, and the timelike flat surface $f(u,v)$ is not (geodesically) complete.
\end{proof}

\end{document}